\documentclass[reqno,11pt]{amsart}
\usepackage{amsmath,amssymb,latexsym,soul,cite,mathrsfs,amsfonts}
\usepackage{color,enumitem,graphicx}
\usepackage[colorlinks=true,urlcolor=blue,
citecolor=red,linkcolor=blue,linktocpage,pdfpagelabels,
bookmarksnumbered,bookmarksopen]{hyperref}
\usepackage[english]{babel}
\usepackage{mwe}
\usepackage{comment}
\usepackage{tikz}
\usepackage[left=2.9cm,right=2.9cm,top=2.8cm,bottom=2.8cm]{geometry}
\usepackage[hyperpageref]{backref}

\usepackage[colorinlistoftodos]{todonotes}
\makeatletter
\providecommand\@dotsep{5}
\def\listtodoname{List of Todos}
\def\listoftodos{\@starttoc{tdo}\listtodoname}
\makeatother

\usepackage{verbatim} 

\numberwithin{equation}{section}

\newtheorem{theorem}{Theorem}[section]
\newtheorem{proposition}[theorem]{Proposition}
\newtheorem{lemma}[theorem]{Lemma}
\newtheorem{corollary}[theorem]{Corollary}
\newtheorem{remark}[theorem]{Remark}

\newcommand\restr[2]{{
  \left.\kern-\nulldelimiterspace 
  #1 
  \vphantom{\big|} 
  \right|_{#2} 
  }}

\title[Critical points with prescribed energy for a class of functionals]
{Critical points with prescribed energy for a class of functionals depending on a parameter: existence, multiplicity and bifurcation results}

\author[H. Ramos Quoirin]{Humberto Ramos Quoirin}
\author[G. Siciliano]{Gaetano Siciliano}

\author[K. Silva]{Kaye Silva}

\address[H. Ramos Quoirin]{\newline \indent CIEM-FaMAF \newline \indent Universidad Nacional de C\'{o}rdoba, 
(5000) C\'{o}rdoba, Argentina}
\email{\href{mailto:humbertorq@gmail.com}{ humbertorq@gmail.com}}

\address[G. Siciliano]{\newline\indent
	Departamento de Matem\'atica - Instituto de Matem\'atica e Estat\'istica
	\newline\indent 
	Universidade de S\~ao Paulo
	\newline\indent
	Rua do Mat\~ao 1010,  05508-090  S\~ao Paulo, SP,  Brazil}
\email{\href{mailto:sicilian@ime.usp.br}{sicilian@ime.usp.br}}

\address[K. Silva]{\newline\indent
	Instituto de Matem\'atica e Estat\'istica.   
	\newline\indent 
	Universidade Federal de Goi\'as,
	\newline\indent
Rua Samambaia, 74001-970, Goi\^ania, GO, Brazil}
\email{\href{mailto:kayesilva@ufg.br}{kayesilva@ufg.br}}


\begin{document}

\begin{abstract}
We look for critical points with prescribed energy for the family of even functionals $\Phi_\mu=I_1-\mu I_2$, where $I_1,I_2$ are $C^1$ functionals on a Banach space $X$, and $\mu \in \mathbb{R}$. For several classes of $\Phi_\mu$ we prove the existence of infinitely many couples $(\mu_{n,c}, u_{n,c})$ such that $$\Phi'_{\mu_{n,c}}(\pm u_{n,c}) = 0 \quad \mbox{and} \quad \Phi_{\mu_{n,c}}( \pm u_{n,c}) = c \quad \forall n \in \mathbb{N}.$$ More generally, we analyze the structure of the solution set of the problem $$\Phi_\mu'(u)=0, \quad \Phi_{\mu}(u)=c$$  with respect to $\mu$ and $c$. In particular, we show that the maps $c \mapsto \mu_{n,c}$ are continuous, which gives rise to a family of {\it energy curves} for this problem. The analysis of these curves provide us with several bifurcation and multiplicity type results, which are then applied to some elliptic problems. Our approach is based on the {\it nonlinear generalized Rayleigh quotient} method developed in \cite{I1}.
\end{abstract}

\bigskip

\maketitle
\begin{center}
\begin{minipage}{12cm}
\tableofcontents
\end{minipage}
\end{center}

\bigskip

\section{Introduction}

A wide and diverse range of problems in nonlinear pdes (among other fields) can be formulated as a critical point equation
\begin{equation}\label{e}\Phi'(u)=0,\end{equation} where $\Phi'$ is the Fr\'echet derivative of a functional $\Phi$ (defined on a suitable function space $X$), the so-called {\it energy functional} of the associated Euler-Lagrange equation. Critical point theory and variational methods have been successfully developed to solve this
kind of equation, which is often coupled to some additional constraint on $u$  (e.g. a {\it sign constraint} $u>0$ or a {\it mass constraint} $\|u\|=m$) and a huge bibliography is available on this subject.

\medskip

 This paper  is devoted to the investigation of \eqref{e}
under  a different constraint, namely, a level (or energy) constraint. 
Motivated by several nonlinear elliptic problems (see below), we shall consider a class of functionals depending on a real parameter, viz. \begin{equation*}
	\Phi_\mu:=I_1-\mu I_2,
\end{equation*}
where $\mu \in \mathbb{R}$, $I_1,I_2 \in C^1(X)$, and $X$ is Banach space  which is assumed to be infinite-dimensional, uniformly convex, and
equipped with $\|\cdot \| \in C^1(X \setminus \{0\})$. 

For a given $c \in \mathbb{R}$ we consider the problem
\begin{equation}\label{ef}
\Phi_\mu'(u)=0, \quad \Phi_{\mu}(u)=c,
\end{equation}
 i.e.  we look for couples $(\mu,u) \in \mathbb{R} \times X \setminus \{0\}$ solving this system.  To this end, we shall follow the {\it nonlinear generalized Rayleigh quotient} method, introduced by Y. Ilyasov \cite{I1}. We intend to show that this method  is suitable to solve \eqref{ef}, and more generally to investigate the structure of the set
 $$\mathcal{S}:=\{(\mu,c) \in \mathbb{R}^2: \Phi_\mu \mbox{ has a critical point at the level } c  \}.$$
  Let us note that some preliminary results on \eqref{ef} can be found in \cite{I3}, as well as in \cite{QSS}, where the case $c=0$ has been treated.

 Under some conditions on the family of functionals $\Phi_\mu$  we shall see that there exist infinitely many pairs $(\mu_{n,c},\pm u_{n,c})$ solving \eqref{ef}. This result, which is proved via the Ljusternik-Schnirelman theory, will be established not only for a single value of $c$, but for any $c$ in an open interval $\mathcal{I}\subset \mathbb{R}$. Thus it makes sense to study  the behaviour  of $\mu_{n,c}$ and $u_{n,c}$ with respect to $c \in \mathcal{I}$. In many cases (including the pdes considered in this work) the values $\mu_{n,c}$ depend continuously on $c$, so that letting $c$ vary  we shall obtain a family of {\it energy curves} $\{(\mu_{n,c},c); c \in \mathcal{I}\}_{n \in \mathbb{N}}$. The properties of these {\it energy curves} shall then be compared with the bifurcation analysis of the unconstrained problem $\Phi_\mu'(u)=0$. In particular, this procedure shall allow us to deduce several bifurcation and multiplicity results for this problem.

A simple but enlightening functional where our abstract framework applies
is 
$$\Phi_\mu(u)=\frac{1}{2}\|\nabla u\|_2^2 - \frac{\mu}{r} \|u\|_r^r - \frac{1}{q}\|u\|_q^q, \quad u \in H_0^1(\Omega),$$
which is associated to the semilinear boundary value problem 
\begin{equation}\label{semi}
\begin{cases}
-\Delta u=\mu |u|^{q-2}u +|u|^{r-2}u &\mbox{ in } \Omega,\\
u=0 &\mbox{ on } \partial \Omega.
\end{cases}
\end{equation}
Here $\Omega \subset \mathbb{R}^N$ ($N \geq 1$) is a bounded domain, $\mu \in \mathbb{R}$, and $1<q \le 2<r<2^*$.

 More generally we consider the $p$-Laplacian problem
 \begin{equation}
 \label{quasi}
 \begin{cases}
 -\Delta_p u=\mu |u|^{q-2}u +|u|^{r-2}u &\mbox{ in } \Omega, \\
 u=0 &\mbox{ on } \partial \Omega,
 \end{cases}
 \end{equation}
 where now $1<q \le p<p^*$ and $1<r<p^{*}, r\neq p$.
  It should be noted that classical bifurcation methods (see e.g. \cite{R2}) are a very powerful tool to treat \eqref{semi} with $q=2$, providing the existence of branches of solutions  bifurcating from $(\mu,u)=(\mu_k,0)$ for any eigenvalue $\mu_k$ of the Dirichlet Laplacian. However, these methods do not apply to \eqref{semi} for $q<2$ (see \cite{BM} for a discussion on this issue), neither to \eqref{quasi}, in both cases $q=p$ and $q<p$. We shall see that our results cover all these variants with the same efficiency, and even though we are not able to obtain continua of solutions, we  establish some  bifurcation results which are consistent with previous results on these problems \cite{ABC,AGP,BM,BW,dPM,D,GP,GT,GV,I2,R1,R2}.
 
 A second example motivating this work is the  Schr\"odinger-Poisson (or Schr\"odinger-Bopp-Podolsky) problem
\begin{equation}
\label{SP}
\left\{
\begin{array}
[c]{lll}%
-\Delta u+\omega u+\mu \phi u=|u|^{p-2}u & \mathrm{in} & \mathbb{R}^3,\\
-\Delta\phi+a^2\Delta ^2u=4\pi u^2  & \mathrm{in} & \mathbb{R}^3,
\end{array}
\right. 
\end{equation}
  where $p\in(2,3)$, $\omega>0$, and $a\ge 0$. We look for radial solutions 
 $u\in H_r^1(\mathbb{R}^3)$ of \eqref{SP}. 
 Let $\|\cdot\|$ be given by $\|u\|^2=\|\nabla u\|_2^2+\omega \|u\|_2^2$, and 
  \begin{equation*}
  	\mathcal{D}_r=\left\{\phi\in D_r^{1,2}(\mathbb{R}^3):\Delta\phi\in L_r^2(\mathbb{R}^3)\right\}.
  \end{equation*}
  It is known that for every $u\in H_r^1(\mathbb{R}^3)$, there exists an unique $\phi_u\in \mathcal{D}_r$ solving the second equation in \eqref{SP}, see \cite{SS}. Moreover, critical points of the functional $\Phi_\mu:H_r^1(\mathbb{R}^3)\to \mathbb{R}$  defined by
  \begin{equation*}
  	\Phi_\mu(u)=\frac{1}{2}\int_{\mathbb{R}^3} |\nabla u|^2+\frac{\omega}{2}\int_{\mathbb{R}^3} |u|^2+\frac{\mu}{4}\int_{\mathbb{R}^3} \phi_uu^2-\frac{1}{p}\int_{\mathbb{R}^3} |u|^p, 
  \end{equation*}
  are classical solutions of \eqref{SP}. 
  
  Although  \eqref{semi} with $q<2$ and \eqref{SP} are very different in nature, their variational structures are somehow symmetric to each other, as one can note in the geometry of their fibering maps.  We shall investigate the structure of \eqref{ef} for this functional and deduce, as a consequence, several multiplicity results extending previous results on this problem contained in \cite{AR,BF,DS,L,Ru,Ru1,S1}.

 \subsection{Main abstract results}
 Before stating our results let us describe how the {\it nonlinear generalized Rayleigh quotient} method applies to
the system \eqref{ef}.
Assume that $I_2(u)\neq 0$ for every $u\in X\setminus\{0\}$ (which is the case in many elliptic problems), so that one can solve the level constraint explicitely in $\mu$:
\begin{equation}\label{mu}
	\Phi_\mu(u)=c \quad \Longleftrightarrow \quad \mu=\mu(c,u):=\frac{I_1(u)-c}{I_2(u)}.
\end{equation}
The key point in this approach is to note that for any $c$ the functional $u \mapsto \mu(c,u)$ satisfies the following relation, which can be easily checked (see also \cite{I3}):
$$\frac{\partial \mu}{\partial u}(c,u)=\frac{\Phi_{\mu(c,u)}'(u)}{I_2(u)}, \quad \forall u \in X  \setminus \{0\}.$$
Here $\frac{\partial \mu}{\partial u}(c,u)$ denotes the Fr\'echet derivative of the functional $u \mapsto \mu(c,u)$.
The above relation provides us with the following equivalence:
\begin{equation*}
\Phi_\mu'(u)=0, \quad \Phi_{\mu}(u)=c \quad \Longleftrightarrow \quad \mu=\mu(c,u),\quad \frac{\partial \mu}{\partial u}(c,u)=0,
\end{equation*}
i.e. \eqref{ef} can be completely solved by looking for critical points of
 the functional $u \mapsto \mu(c,u)$. In other words, \eqref{ef} is solvable if and only if $\mu$ is a critical value (and $u$ an associated critical point) of the latter functional. Thus, denoting by $\mathcal{K}(c)$ the set of critical values of the functional  $u \mapsto \mu(c,u)$, we  immediately find a sufficient and necessary condition for the solvability of \eqref{ef}:
 
\begin{theorem}\label{THM0}
For a given $c \in \mathbb{R}$ the problem \eqref{ef} has a solution 
$(\mu, u)$ if, and only if, $\mu \in \mathcal{K}(c)$ and $u$ is the associated critical point. In particular, if $u \mapsto \mu(c,u)$ has a ground state (or least energy) level $GS(c)$ then $\eqref{ef}$ has no solution for $\mu<GS(c)$.

\end{theorem}

Let us focus now on finding points in $\mathcal{K}(c)$ by following the nonlinear generalized Rayleigh quotient method developed in \cite{I1,I3}. We consider the fibering map associated to the functional $u \mapsto \mu(c,u)$, namely, the real-valued function $\psi_{c,u}$ given by 
\begin{equation}\label{mut}
\psi_{c,u}(t):=\mu(c,tu)=\frac{I_1(tu)-c}{I_2(tu)}, \quad t>0
\end{equation} 
for any fixed $(c,u) \in \mathbb{R} \times X\setminus\{0\}$. 
Let us assume the following condition:\\

\begin{enumerate}[label=(H\arabic*),ref=H\arabic*,start=1]
	\item \label{H1}  There exists an open set $\mathcal{I} \subset \mathbb{R}$ such that:
	\begin{enumerate}
\item the map  $(c,u,t) \mapsto \psi'_{c,u}(t)$ belongs to $C^1(\mathcal{I} \times X \setminus\{0\} \times (0,\infty))$;
\item for every $(c,u) \in \mathcal{I} \times X \setminus\{0\}$ the map $\psi_{c,u}$ has exactly one local minimizer $t^+(c,u)>0$ of Morse type or for every $(c,u) \in \mathcal{I} \times X \setminus\{0\}$ the map $\psi_{c,u}$ has exactly one local maximizer $t^-(c,u)>0$ of Morse type. \\
\end{enumerate}
\end{enumerate}

We shall see  some applications where both possibilities in (b) will occur.

For the sake of simplicity, let us assume for the moment that the first possibility occurs (the case where the second one holds is similar), and set $t(c,u):=t^+(c,u)$.
We introduce the reduced functional  $\Lambda \in C^1(\mathcal{I} \times X \setminus \{0\})$ given by
\begin{equation}\label{eq:Lambda}
	\Lambda(c,u):=\psi_{c,u}(t(c,u))=\mu(c,t(c,u)u).
\end{equation} 
For any $c \in \mathcal{I}$ the functional $u \mapsto \Lambda(c,u)$ turns out to be $0$-homogeneous, and one may then look for critical points of this functional by dealing with $\widetilde{\Lambda}$, the restriction of $\Lambda$ to $\mathcal{I}\times S$, where $S$ is the unit sphere in $X$.

Let us assume that $I_1,I_2$ are even, so that $\mu(c,\cdot)$, $t(c,\cdot)$ and $\Lambda(c,\cdot)$ are even as well. Our aim is to apply the Ljusternick-Schnirelmann theory to the functional $u \mapsto \widetilde{\Lambda}(c,u)$. To this end, let us recall that
given a nonempty symmetric and closed set $F \subset S $, the Krasnoselskii genus of $F$ is given by $$\gamma(F):=\inf\{n \in \mathbb{N}: \exists h:F\to  \mathbb{R}^n \setminus \{0\} \mbox{ odd and continuous} \}.$$   For every $n \in \mathbb{N}$ we set
\begin{equation}\label{dfn}
\mathcal{F}_n:=\{F\subset S: F \mbox{ is compact, symmetric, and } \gamma(F)\ge n\}. 
\end{equation}
We shall assume the following condition, which contains two alternatives in accordance with the behavior of $u \mapsto \widetilde{\Lambda}(c,u)$:\\

\begin{enumerate}[label=(H\arabic*),ref=H\arabic*,start=2]
	\item\label{H2} The functional $u \mapsto \widetilde{\Lambda}(c,u)$ is bounded from below (respect. from above) and  satisfies the Palais-Smale condition at the level $\mu_{n,c}:=\displaystyle
	\inf_{F\in \mathcal{F}_n}\sup_{u\in F}\Lambda(c,u)$ (respect. at the level $\mu_{n,c}:=\displaystyle
	\sup_{F\in \mathcal{F}_n}\inf_{u\in F}\Lambda(c,u)$) for every $n \in \mathbb{N}$.\\
\end{enumerate}

Recall that  the functional $u \mapsto \widetilde{\Lambda}(c,u)$ satisfies the Palais-Smale condition at the level $\mu$ if any sequence $\{u_k\} \subset S$ such that $\widetilde{\Lambda}(c,u_k) \to \mu$ and $\frac{\partial \widetilde{\Lambda}}{\partial u} (c,u_k) \to 0$ has a convergent subsequence.

The assumptions \eqref{H1} and \eqref{H2} will be checked for the functionals associated to \eqref{quasi} and \eqref{SP}. Let us anticipate that for \eqref{quasi} the first case in \eqref{H2} holds, whereas the second one occurs for \eqref{SP}.

We are now in position to state our main abstract result. This one extends  \cite[Theorem 1.1]{QSS}, which deals with the case $c=0$, and provides the following informations on the sequence $\{\mu_{n,c}\}$:

\begin{theorem}\label{THM1} 
Suppose that \eqref{H1} holds and  let $c \in \mathcal{I}$.
\begin{enumerate}[label=(\arabic{*}), ref=\arabic{*}]

\item \label{THM1-2} 
Assume that $t(c,u)$ is the only critical point of $\psi_{c,u}$, for every $u \in X \setminus \{0\}$.  If $u \mapsto \Lambda(c,u)$ is bounded from below (respect. above) and $\mu<\mu_{1,c}= \displaystyle \inf_{u \in X \setminus \{0\}} \Lambda(c,u)$ (respect. $\mu>\mu_{1,c}= \displaystyle \sup_{u \in X \setminus \{0\}} \Lambda(c,u)$) then there exists no $u \in X \setminus \{0\}$ such that  
 $$\Phi'_\mu(u)=0\qquad\text{and}\qquad\Phi_\mu(u)=c.$$

\item\label{THM1-1} If  \eqref{H2} holds then there exist infinitely many  $u_{n,c} \in X\setminus\{0\}$ such that $$\Phi'_{\mu_{n,c}}(\pm u_{n,c}) = 0 \quad \mbox{and} \quad \Phi_{\mu_{n,c}}( \pm u_{n,c}) = c \quad \forall n \in \mathbb{N}.$$ If, in addition, $\Lambda(c,w_n) \to +\infty$ whenever $w_n \rightharpoonup 0$ in $X$, then $\mu_{n,c}\to + \infty$ as $n \to+ \infty$.\\

\end{enumerate}
\end{theorem}

	

\begin{remark}\label{infsup}If \eqref{H2} holds with $u \mapsto \Lambda(c,u)$ bounded from below (respect. above)
then $\{\mu_{n,c}\}$ is nondecreasing (respect nonincreasing). Moreover, in the second case we have $\mu_{n,c}=-\displaystyle \inf_{F\in \mathcal{F}_n}\sup_{u\in F}(-\Lambda(c,u))$, 
as well as $\mu_{n,c}=\left(\displaystyle \inf_{F\in \mathcal{F}_n}\sup_{u\in F}(\Lambda(c,u))^{-1}\right)^{-1}$ if, in addition, $\Lambda(c,u)$ is positive on $S$.
These characterizations will provide us at least two possible behaviors for $\{\mu_{n,c}\}$ as $n \to +\infty$: $\mu_{n,c} \to 0$ or $\mu_{n,c} \to -\infty$, cf. Theorem \ref{THMAP3} below.
\end{remark}

  \begin{figure}[h]
	\centering
	
	\begin{tikzpicture}[>=latex]
	\draw[->] (-2,0) -- (4,0);

	\draw [thick] (1,-.1) node[below]{\scalebox{0.8}{$\mu_{1,c}$}} -- (1,0.05); 
	\draw [thick] (1.8,-.1) node[below]{\scalebox{0.8}{$\mu_{k,c} $}} -- (1.8,0.05); 
	\draw [thick] (3,-.1) node[below]{\scalebox{0.8}{$\mu_{n,c} $}} -- (3,0.05); 
	\draw  (0,.1) node[above]{\scalebox{1.5}{$\nexists$}} ; 
	\draw [thick] (2.4,-0.1) node[below]{$\cdots$} -- (2.4,-0.1); 
	\draw [thick] (3.7,-0.1) node[below]{\scalebox{0.8}{$\ \to+ \infty$}} -- (3.7,-0.1); 
	
	\end{tikzpicture}
	\caption{The sequence $\{\mu_{n,c}\}$ provided by Theorem \ref{THM1}.} \label{fig0}
\end{figure}
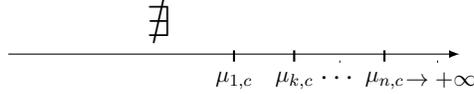

\medskip
\subsection{Energy curves}
A natural question related to Theorem \ref{THM1} is the behavior of $\mu_{n,c}$ with respect to $c$. The following conditions on $\Lambda$ shall provide the continuity of $c \mapsto \mu_{n,c}$, which gives rise to the family of {\it energy curves} $\{(\mu_{n,c},c); c \in \mathcal{I}\}_{n \in \mathbb{N}}$:\\

\begin{enumerate}[label=(H\arabic*),ref=H\arabic*,start=3]
	\item\label{H3} For any $u \in S$ the map $c \mapsto \Lambda(c,u)$ is decreasing (respect. increasing) in $\mathcal{I}$. In addition, $\Lambda$ is bounded from above (respect. below) in any compact set $K \subset \mathcal{I} \times S$.
	 Finally, if $\Lambda$ is bounded in
	  $[a,b]\times S_0 \subset \mathcal{I} \times S$  then $\frac{\partial \Lambda}{\partial c}$ is also bounded and away from zero therein.\\
\end{enumerate}

The following generalized Palais-Smale condition shall be required as well:
\begin{enumerate}[label=(PSG),ref=PSG]
	\item\label{PSG}
			If $c_n\to c\in \mathcal{I}$ and $\{u_{n}\}\subset S$ are such that $\{\Lambda(c_n,u_n)\}$ is bounded and $\frac{\partial \widetilde{\Lambda}}{\partial u}(c_n,u_n)\to 0$,  then $\{u_n\}$ has  a convergent subsequence.\\
\end{enumerate}

We shall see that \eqref{H3} is satisfied if $I_1$ is bounded on bounded sets of $X$, $I_2$ is weakly continuous in $X \setminus \{0\}$ and $t(c,u)$ is bounded and bounded away from zero in $[a,b] \times S_0 \subset \mathcal{I} \times S$, for any $S_0$ away from zero in the weak topology of $X$.

The next result combined with the asymptotics of the maps $c \mapsto \mu_{n,c}$
 provide us with some informations on the structure of the set $$\mathcal{S}:=\{(\mu,c) \in \mathbb{R}^2: \Phi_\mu \mbox{ has a critical point at the level } c  \}.$$


\begin{theorem}\label{THM2}
Assume \eqref{H1}, \eqref{H2}, \eqref{H3} and \eqref{PSG}. Then for every $n \in \mathbb{N}$ the map $c \mapsto \mu_{n,c}$ is decreasing (respect. increasing) and locally Lipschitz continuous in $\mathcal{I}$.
\end{theorem}

\begin{remark}
The statement of Theorem \ref{THM2} has to be understood in the sense that 
\begin{itemize}
\item the map $c \mapsto \mu_{n,c}$ is decreasing if in \eqref{H2} the functional $u\mapsto\widetilde\Lambda(c,u)$
is bounded below and in \eqref{H3}  the map $c \mapsto \Lambda(c,u)$ is decreasing.

\item the map $c \mapsto \mu_{n,c}$ is increasing if in \eqref{H2} the functional $u\mapsto\widetilde\Lambda(c,u)$
is bounded above and in \eqref{H3}  the map $c \mapsto \Lambda(c,u)$ is increasing.

\end{itemize}
\end{remark}

\begin{remark} One may check that the continuity and the decreasing behavior of $c \mapsto \mu_{1,c}$ holds under weaker conditions. Indeed, assume that $u \mapsto \Lambda(c,u)$ is bounded from below and $\mu_{1,c}=\displaystyle \inf_{u \in X \setminus \{0\}} \Lambda(c,u)$ is achieved for every $c \in \mathcal{I}$. If for every $u \neq 0$ the map $c \mapsto \Lambda(c,u)$ is decreasing in $\mathcal{I}$, then $c \mapsto \mu_{1,c}$ is decreasing in $\mathcal{I}$. Moreover, if for every $u \neq 0$ the map $c \mapsto \Lambda(c,u)$ is concave in $\mathcal{I}$, then $c \mapsto \mu_{1,c}$ is concave in $\mathcal{I}$ as well, and therefore continuous. Finally, let us assume that $t(c,u)$ is the only critical point of $\psi_{c,u}$, for every $u \in X \setminus \{0\}$. Then the solution $u_{1,c}$ with $\mu=\mu_{1,c}$ is the ground state solution of the problem $\Phi_\mu'(u)=0$. In other words, whenever achieved,
the ground state level of $\Phi_\mu$ is the value $c$ such that $\mu=\mu_{1,c}$.
\end{remark}

 Finally, we point out that besides \eqref{quasi} and \eqref{SP} our abstract results apply to many other elliptic models.
 In particular, one may deal with systems of equations, fourth-order problems, Kirchhoff type equations, $(p,q)$-Laplacian problems, etc. Furthermore, our variational setting is general enough so as to apply to equations subject to other kind of boundary conditions.
 
\subsection*{Outline of the paper}

 In Section \ref{sec:app} we describe the applications of Theorems \ref{THM0}, \ref{THM1} and \ref{THM2}  respectively to the problems \eqref{quasi} with $q=p$,
 \eqref{quasi} with $q<p$,  and \eqref{SP}. Our results on the set $\mathcal{S}$ for these problems are illustrated in Figures \ref{fig1}, \ref{fig2}, \ref{fig:CC}, and \ref{fig:SP}. 
In Section \ref{sec:proofs} we prove Theorems \ref{THM1} and \ref{THM2}.
 Lastly, in Sections
 \ref{sec:4},  \ref{SubCC} and  \ref{sec:6} we apply our abstract results to three classes of functionals, which include in particular those associated to problems \eqref{quasi} with $q=p$,
 \eqref{quasi} with $q<p$  and  \eqref{SP}. We analyze the asymptotic behaviour of the maps $c \mapsto \mu_{n,c}$.
 A final Appendix includes a simple result frequently used in the paper.

\subsection*{Notation} Throughout the paper, we use the following notations:

\begin{itemize}
	\item unless otherwise stated $\Omega$ denotes a bounded domain of $\mathbb{R}^N$ with $N\geq 1$;
	
	\item given $r>1$, we denote by $\Vert\cdot\Vert_{r}$ the usual norm in
	$L^{r}(\Omega)$, and by $r^*$ the critical Sobolev exponent, i.e. $r^*=\frac{Nr}{N-r}$ if $r<N$ and $r^*=+\infty$ if $r \geq N$;
	
	\item strong and weak convergences are denoted by $\rightarrow$ and
	$\rightharpoonup$, respectively;
	
	\item given $f\in L^{1}(\Omega)$, we set $f^{\pm}:=\max(\pm f,0)$. The integral $\int_{\Omega}f$ is considered
	with respect to the Lebesgue measure. Equalities and inequalities involving
	$f$ shall be understood holding \textit{a.e.};
	
	\item $S$ is the unit sphere in the Banach space $X$;
	\item if $X=W_0^{1,p}(\Omega)$ then $\|u\|=\|\nabla u\|_p$ for any $u \in X$.
	
\end{itemize}
\medskip
\section{Applications to elliptic problems} \label{sec:app}

As already mentioned, we shall apply the abstract theorems 
of the previous secton
in three interesting cases. The proofs of the results of this section are postponed
to Sections \ref{sec:4},  \ref{SubCC}, and  \ref{sec:6} respectively,
since the related functionals are particular cases of abstract functionals, see  Theorems \ref{ts1},  \ref{ts2}
and  \ref{THMSPAbstrac}, respectively.

\subsection{Nonhomogeneous perturbations of an eigenvalue problem}We consider the problem \eqref{quasi} with $q=p$ and $1<r<p^*$, $r \neq p$. The energy functional associated to this problem is given by 
\begin{equation*}
	\Phi_\mu(u)=\frac{1}{p}\left(\|\nabla u\|_p^p-\mu\|u\|_p^p\right)-\frac{1}{r} \|u\|_r^r, \quad u \in W_0^{1,p}(\Omega).
\end{equation*} 

Let us recall that weak solutions of \eqref{quasi} are locally Holder continuous.
We shall consider the superhomogeneous case $p<r$ (or superlinear if $p=2$) as well as the subhomogeneous one $p>r$ (the sublinear case if $p=2$). 

Let us set, for every $n \in \mathbb{N}$, $$\mu_{n}=\inf_{F\in \mathcal{F}_n}\sup_{u\in F}\frac{\|\nabla u\|_p^p}{\|u\|_p^p},$$
where $\mathcal{F}_n$ is given by \eqref{dfn}. It is well-known that $\{\mu_n\}$ is a nondecreasing sequence of eigenvalues of the Dirichlet p-Laplacian, which provides all  eigenvalues of $(-\Delta,H_0^1(\Omega))$ if $p=2$.

Theorems \ref{THM0} and \ref{THM1} yield the following result:

\begin{theorem}\label{THMAP1} Under the previous conditions let $c>0$ (respect. $c<0$) if $p <r$ (respect. $p>r$). Then there exist infinitely many $(\mu_{n,c},u_{n,c}) \in   \mathbb{R} \times W_0^{1,p}(\Omega)\setminus\{0\}$ such that $\Phi_{\mu_{n,c}}(\pm u_{n,c})=c$ and $\Phi'_{\mu_{n,c}}(\pm u_{n,c})=0,$ 
i.e. $\pm u_{n,c}$ are weak solutions of \eqref{quasi} with $\mu=\mu_{n,c}$, having energy $c$, for every $n$. Moreover:
\begin{enumerate}
\item $\{\mu_{n,c} \}$ is non-decreasing and  $\mu_{n,c},\|u_{n,c}\| \to +\infty$ as $n\to +\infty$. 	\smallskip
\item \eqref{quasi} has no weak solution having energy $c$ for $\mu<\mu_{1,c}$. 	\smallskip
\item $u_{n,c}$ is sign-changing for $n$ large enough.
\end{enumerate}
\end{theorem}

Theorem \ref{THM2} and some asymptotic analysis 
made in Theorem \ref{ts1} give the behaviour of the map
$c \mapsto \mu_{n,c}$ and an existence result.

\begin{theorem}\label{EC1}
	Under the assumptions of Theorem \ref{THMAP1}, the following properties hold: for every $n\in \mathbb{N}$ the map $c \mapsto \mu_{n,c}$ is locally Lipschitz continuous and decreasing in $(0,+\infty)$ (respect. $(-\infty,0)$) if $p<r$ (respect. $p>r$). Furthermore:
	\begin{enumerate} 
		\item If $p<r$ then (see Figure \ref{fig1}):
		\begin{enumerate}
			\item  $\mu_{n,c}\to -\infty$ and $\|u_{n,c}\| \to +\infty$ as $c\to+ \infty$, i.e. $(-\infty,+\infty)$ is a bifurcation point.
			\item   $\mu_{n,c}\to \mu_n$ and $u_{n,c} \to 0$ in $W_0^{1,p}(\Omega)$ as $c \to 0^+$, i.e. every $(\mu_n,0)$ is a bifurcation point.
			\item For every $\mu\in \mathbb{R}$ the problem \eqref{quasi} has infinitely many pairs of weak solutions $\pm u_n\in W_0^{1,p}(\Omega)$ with positive energy. Moreover $\Phi_\mu(u_n)\to +\infty$ and $\|u_n\|\to +\infty$ as $n \to +\infty$, so $(\mu,+\infty)$ is a bifurcation point for any  $\mu\in \mathbb{R}$.
		\end{enumerate}
		\item If $p>r$ then (see Figure \ref{fig2}):
		\begin{enumerate}
			\item $\mu_{n,c}\to -\infty$ and $u_{n,c} \to 0$ in $X$ as $c\to 0^-$, i.e. $(-\infty,0)$ is a bifurcation point.
			\item $\mu_{n,c}\to \mu_n$ and $\|u_{n,c}\| \to +\infty$ as $c \to -\infty$, i.e. every $(\mu_n,+\infty)$ is a bifurcation point.
			\item For every $\mu\in \mathbb{R}$ the problem \eqref{quasi} has infinitely many pairs of weak solutions $\pm v_n\in W_0^{1,p}(\Omega)$ with negative energy. Moreover $\Phi_\mu(v_n)\to 0$ and $v_n\to 0$ in $W_0^{1,p}(\Omega)$  as $n \to +\infty$, so $(\mu,0)$ is a bifurcation point for any $\mu \in \mathbb{R}$.
		\end{enumerate}
	\end{enumerate}
\end{theorem}

\begin{figure}[h]
	\centering
	\begin{tikzpicture}[>=latex]
	\draw[->] (-1,0) -- (6,0) node[below] {\scalebox{0.8}{$\mu$}};
	\foreach \x in {}
	\draw[shift={(\x,0)}] (0pt,2pt) -- (0pt,-2pt) node[below] {\footnotesize $\x$};
	\draw[->] (0,-.5) -- (0,3) node[left] {\scalebox{0.8}{$c$}};
	\foreach \y in {}
	\draw[shift={(0,\y)}] (2pt,0pt) -- (-2pt,0pt) node[left] {\footnotesize $\y$};
	\draw[blue,thick] (-2.5,1.2) .. controls (0,1) and (1,1) .. (1.5,0);
	\draw[blue,thick] (-2.5,1.7) .. controls (0,1.5) and (1.5,1.5) .. (2,0);
	\draw [thick] (1.5,-.1) node[below]{\scalebox{0.8}{$\mu_{1}$}} -- (1.5,0.05); 
	\draw [thick] (1.8,0) -- (1.8,3);
	
	\draw [thick] (2,-.1) node[below]{\scalebox{0.8}{$\mu_{2}$}} -- (2,0.05); 
	\draw [thick] (2.5,-.1) node[below]{\scalebox{0.8}{$\mu_{3}$}} -- (2.5,0.05); 
	\draw [thick] (3.3,-.1) node[below]{\scalebox{0.8}{$\mu_{n}$}} -- (3.3,0.05); 
	
	\draw [thick] (2.8,0) node[above]{$\cdots$}; 
	\draw[blue,thick] (-2.5,2.2) .. controls (0,2) and (2,2) .. (2.5,0);
	\draw[blue,thick] (-2.5,2.7) .. controls (0,2.5) and (2.5,2.5) .. (3.3,0);
	
	\draw [thick] (-.1,2.3) node[right]{$\vdots$} ;
	\draw (1.8,.45) node{\scalebox{0.8}{$\bullet$}};
	\draw (1.8,1.2) node{\scalebox{0.8}{$\bullet$}};
	\draw (1.8,1.9) node{\scalebox{0.8}{$\bullet$}};
					\draw  (-1,.1) node[above]{\scalebox{1.5}{$\nexists$}}  ; 	
	\draw  (-2.8,1) node[above]{\scalebox{0.8}{$\mu_{1,c}$}} ; 	
	\draw  (-2.8,1.5) node[above]{\scalebox{0.8}{$\mu_{2,c}$}} ; 
	\draw  (-2.8,2) node[above]{\scalebox{0.8}{$\mu_{3,c}$}} ; 
	
	\draw  (-2.8,2.5) node[above]{\scalebox{0.8}{$\mu_{n,c}$}} ; 
	
	\draw  (2.3,2.4) node[above]{\scalebox{0.8}{$\mu=\overline{\mu}$}} ;
	\end{tikzpicture}
	\caption{Energy curves for \eqref{quasi} with $q=p<r$. Here we assume that $ \mu_{n,c}$ is increasing in $n$, for every $c$. No solution exists below the first curve.} \label{fig1}
\end{figure}
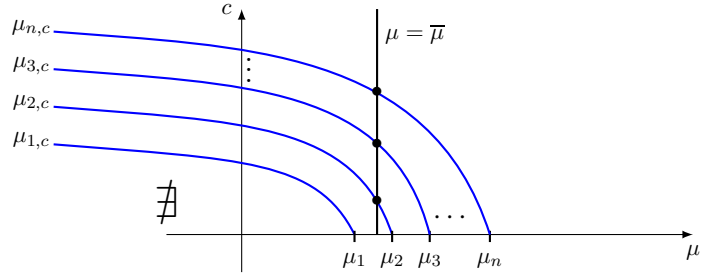

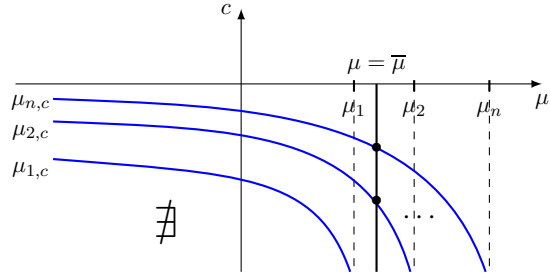
\begin{figure}[h]
	\centering
	\begin{tikzpicture}[>=latex]
	\draw[->] (-3,0) -- (4,0) node[below] {\scalebox{0.8}{$\mu$}};
	\foreach \x in {}
	\draw[shift={(\x,0)}] (0pt,2pt) -- (0pt,-2pt) node[below] {\footnotesize $\x$};
	\draw[->] (0,-2.5) -- (0,1) node[left] {\scalebox{0.8}{$c$}};
	\foreach \y in {}
	\draw[shift={(0,\y)}] (2pt,0pt) -- (-2pt,0pt) nde[left] {\footnotesize $\y$};
	\draw[blue,thick] (-2.5,-1) .. controls (0,-1.2) and (1,-1.2) .. (1.45,-2.5);
	\draw [thick] (1.5,-.1) node[below]{\scalebox{0.8}{$\mu_{1}$}} -- (1.5,0.05); 
	\draw [dashed] (1.5,0) -- (1.5,-2.5);
	\draw [thick] (2.3,-.1) node[below]{\scalebox{0.8}{$\mu_{2}$}} -- (2.3,0.05); 
	\draw [dashed] (2.3,0) -- (2.3,-2.5);
	\draw [thick] (3.3,-.1) node[below]{\scalebox{0.8}{$\mu_{n}$}} -- (3.3,0.05); 
	\draw [dashed] (3.3,0) -- (3.3,-2.5);
	
	\draw [thick] (2.4,-2) node[above]{$\cdots$}; 
	\draw[blue,thick] (-2.5,-.5) .. controls (0,-.6) and (1.7,-.6) .. (2.25,-2.5);
	\draw[blue,thick] (-2.5,-.2) .. controls (0,-.3) and (2.5,-.3) .. (3.25,-2.5);
	
	\draw [thick] (1.8,0) -- (1.8,-2.5);
	
	\draw (1.8,-1.55) node{\scalebox{0.8}{$\bullet$}};
	\draw (1.8,-.85) node{\scalebox{0.8}{$\bullet$}};
							\draw  (-1,-1.5) node[below]{\scalebox{1.5}{$\nexists$}}  ; 	
	\draw  (-2.8,-1.4) node[above]{\scalebox{0.8}{$\mu_{1,c}$}} ; 	
	\draw  (-2.8,-.9) node[above]{\scalebox{0.8}{$\mu_{2,c}$}} ; 
	
	\draw  (-2.8,-.5) node[above]{\scalebox{0.8}{$\mu_{n,c}$}} ; 
	
	\draw  (1.8,0) node[above]{\scalebox{0.8}{$\mu=\overline{\mu}$}} ;
	\end{tikzpicture}
	\caption{Energy curves for \eqref{quasi} with $q=p>r$} \label{fig2}
\end{figure}

The results of Theorems \ref{THMAP1} and \ref{EC1} are consistent with the bifurcation picture known for \eqref{quasi} with $q=p$. In the semilinear case $p=2$, it is known that any $(\mu_n,0)$ is a bifurcation point if $r>2$, cf. \cite{R2}, whereas $(\mu_n,+\infty)$ is a bifurcation point  if $r<2$, see \cite{R1,T,DH}. For the $p$-Laplacian case with $p \neq 2$ we refer to \cite{dPM,D,GV} for some results in the case $N=1$ as well as \cite{dPM,D} and  \cite{GT} for some bifurcation results from $(\mu_1,0)$ for $r>p$, and $(\mu_1,+\infty)$ for $r<p$, respectively. See also \cite{BH} for a bifurcation result on a two parameters problem similar to \eqref{quasi} with $r>p$.

Figure \ref{fig1} depicts the energy curves $\{(\mu_{n,c},c): c>0\}$ for \eqref{quasi} with $q=p<r$. In this case, it is assumed that $\mu_{n,c}$ is increasing in $n$ for every $c>0$, so that the curves do not meet each other. However, we can not exclude a picturen as in Figure \ref{fig:ou}, where  $\mu_{k,c}=\mu_{k+1,c}$ for some $k$ and $c$. Likewise, we can not exclude this kind of phenomenon for \eqref{quasi} with $q=p>r$, as well as for the problems considered below.


\begin{figure}[h]
	\centering

		\begin{tikzpicture}[>=latex]
		\draw[->] (-1,0) -- (4,0) node[below] {\scalebox{0.8}{$\mu$}};
		\foreach \x in {}
		\draw[shift={(\x,0)}] (0pt,2pt) -- (0pt,-2pt) node[below] {\footnotesize $\x$};
		\draw[->] (0,-1) -- (0,2.2) node[left] {\scalebox{0.8}{$c$}};
		\foreach \y in {}
		\draw[shift={(0,\y)}] (2pt,0pt) -- (-2pt,0pt) node[left] {\footnotesize $\y$};
		\draw[blue,thick] (.5,.8) .. controls (1.5,.8) and (2,.4) .. (2,0);
		\draw[blue,thick] (-1,1.6) .. controls (0,1.5) and (1,.4) .. (1,0);
		\draw [thick] (1,-.1) node[below]{\scalebox{0.8}{$\mu_{k}$}} -- (1,0.05); 
		\draw [thick] (2,-.1) node[below]{\scalebox{0.8}{$\mu_{k+1}$}} -- (2,0.05); 

		\end{tikzpicture}
		\caption{Energy curves for \eqref{quasi} with $q=p<r$. Here we assume that $\mu_{k,c}=\mu_{k+1,c}$ for some $k \in \mathbb{N}$ and $c>0$.} \label{fig:ou}
\end{figure}
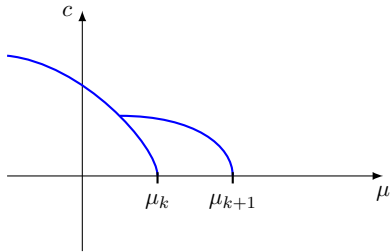

\subsection{A concave-convex problem}
Let us consider now \eqref{quasi} 
with $1<q<p<r<p^*$, and deal with the functional
\begin{equation*}
	\Phi_\mu(u)=\frac{1}{p}\|\nabla u\|_{p}^{p}-\frac{\mu}{q}\|u\|_{q}^q-\frac{1}{r}\|u\|_{r}^r, \quad u \in W_0^{1,p}(\Omega).
\end{equation*}  

Unlike in the case $q=p$, for some values of $c$ the fibering map $\psi_{c,u}$ has now two critical points, which produce then  two sequences of couple of solutions  for these levels:

\begin{theorem}\label{THMAP2} Under the previous conditions there exists $c^*<0$ such that: 
	\begin{enumerate}
		\item For any $c>c^*$ there there exist infinitely many $(\mu_{n,c}^-,u_{n,c}) \in   \mathbb{R} \times W_0^{1,p}(\Omega)\setminus\{0\}$ such that $\Phi_{\mu_{n,c}^-}(\pm u_{n,c})=c$ and $\Phi'_{\mu_{n,c}^-}(\pm u_{n,c})=0$, i.e., $\pm u_{n,c}$ are weak solutions of \eqref{quasi} with $\mu=\mu_{n,c}^-$, having energy $c$, for every $n$.  Moreover:
		\begin{enumerate}
			\item $\{\mu_{n,c}^-\}$ is non-decreasing, and $\mu_{n,c}^-,\|u_{n,c}\| \to +\infty$ as $n\to +\infty$.
			\item $u_{n,c}$ is sign-changing for $n$ large enough.
			\item If $c>0$ and $\mu<\mu_{1,c}^-$ then \eqref{quasi} has no weak solution having energy $c$.\\
		\end{enumerate}

		\item For any $c\in (c^*,0)$  there exist infinitely many $(\mu_{n,c}^+,v_{n,c}) \in   (0,\infty) \times W_0^{1,p}(\Omega)\setminus\{0\}$ such that $\Phi_{\mu_{n,c}^+}(\pm v_{n,c})=c$ and $\Phi'_{\mu_{n,c}^+}(\pm v_{n,c})=0$, i.e., $\pm v_{n,c}$ are weak solutions of \eqref{quasi} with $\mu=\mu_{n,c}^+$, having energy $c$, for every $n$. Moreover:
		\begin{enumerate}
		\item $\{\mu_{n,c}^+\}$ is non-decreasing, and $\mu_{n,c}^+ \to +\infty$ and $v_{n,c} \rightharpoonup 0$ as $n\to +\infty$.
		\item $\mu_{n,c}^+<\mu_{n,c}^-$ for every $n$.
		\item $v_{n,c}$ is sign-changing for $n$ large enough.
		\item If $\mu<\mu_{1,c}^+$ then \eqref{quasi} has no weak solution having energy $c$.
		\end{enumerate}
	\end{enumerate}
\end{theorem}
\medskip

The main properties of the maps $c \mapsto \mu_{n,c}^{\pm}$ are gathered in the next result, which include, as a consequence, the existence of infinitely many solutions for \eqref{quasi} with $\mu$ fixed.

\medskip
\begin{theorem}\label{EC2}
	Under the assumptions of Theorem \ref{THMAP2} the following facts hold (see Figure \ref{fig:CC}):
	\begin{enumerate}
		\item For every $n\in \mathbb{N}$ the map $c \mapsto \mu_{n,c}^-$ is locally Lipschitz continuous and decreasing in $(c^*,+\infty)$, and $\displaystyle \lim_{c\to +\infty}\mu_{n,c}^-=-\infty$. 
		\item For every $n\in \mathbb{N}$ the map $c \mapsto \mu_{n,c}^+$ is locally Lipschitz continuous and decreasing in $(c^*,0)$, and $\displaystyle \lim_{c\to 0^-}\mu_{n,c}^+=0$.

		\item For every $\mu\in \mathbb{R}$ the problem \eqref{quasi} has infinitely many pairs of weak solutions $\pm u_n\in W_0^{1,p}(\Omega)$ with positive energy. Moreover $\Phi_\mu(u_n)\to +\infty$ and $\|u_n\|\to +\infty$ as $n \to \infty$, so $(\mu,+\infty)$ is a bifurcation point for any  $\mu\in \mathbb{R}$.
		\item For every $\mu>0$ the problem \eqref{quasi} has infinitely many pairs of weak solutions $\pm v_n\in W_0^{1,p}(\Omega)$ with negative energy. Moreover $\Phi_\mu(v_n)\to 0$ and $v_n\to 0$ in $W_0^{1,p}(\Omega)$  as $n \to +\infty$, so $(\mu,0)$ is a bifurcation point for any $\mu>0$.
	\end{enumerate}	
\end{theorem}

\begin{figure}[h]
	\centering
	\begin{tikzpicture}[>=latex,scale=0.97]
	\draw[->] (-1,0) -- (6,0) node[below] {\scalebox{0.8}{$\mu$}};
	\foreach \x in {}
	\draw[shift={(\x,0)}] (0pt,2pt) -- (0pt,-2pt) node[below] {\footnotesize $\x$};
	\draw[->] (0,-2) -- (0,3) node[left] {\scalebox{0.8}{$c$}};
	\foreach \y in {}
	\draw[shift={(0,\y)}] (2pt,0pt) -- (-2pt,0pt) node[left] {\footnotesize $\y$};
	\draw[blue,thick] (-2.5,1.2) .. controls (0,1) and (1,1) .. (2,0);
	\draw[blue,thick] (2,0) .. controls (3,-1.4) .. (3,-1.4);
	\draw[blue,thick] (-2.5,1.7) .. controls (0,1.5) and (1.5,1.5) .. (3,0);
	\draw[blue,thick] (3,0) .. controls (4,-1.4) .. (4,-1.4);
	
	\draw [thick] (3.3,0) node[above]{$\cdots$} -- (3.3,0); 
	\draw[blue,thick] (-2.5,2.2) .. controls (0,2) and (2,2) .. (4,0);
	\draw[blue,thick] (4,0) .. controls (5,-1.4) .. (5,-1.4);
	\draw[red,thick] (0,0) .. controls (1,-.1) .. (2.5,-1.4);
	\draw[red,thick] (0,0) .. controls (1.5,-.1) .. (3.3,-1.4);
	\draw[red,thick] (0,0) .. controls (2,-0.1) .. (4.3,-1.4);
	\draw [thick] (3.2,-1.3) node[above]{$\vdots$} -- (3.2,-1.3); 
	\draw [thick] (-.1,1.85) node[right]{$\vdots$} -- (-.1,1.85);
	\draw [thick] (1.8,-2) -- (1.8,3);
	\draw (1.8,.18) node{\scalebox{0.8}{$\bullet$}};
	\draw (1.8,.91) node{\scalebox{0.8}{$\bullet$}};
	\draw (1.8,1.54) node{\scalebox{0.8}{$\bullet$}};
	\draw (1.8,-.18) node{\scalebox{0.8}{$\bullet$}};
	\draw (1.8,-.37) node{\scalebox{0.8}{$\bullet$}};
	\draw (1.8,-.82) node{\scalebox{0.8}{$\bullet$}};
	\draw  (-1,0) node[above]{\scalebox{1.5}{$\nexists$}}  ; 	
	\draw  (-2.8,1) node[above]{\scalebox{0.8}{$\mu_{1,c}^-$}} ; 	
	\draw  (-2.8,1.5) node[above]{\scalebox{0.8}{$\mu_{2,c}^-$}} ; 
	\draw  (-2.8,2) node[above]{\scalebox{0.8}{$\mu_{n,c}^-$}} ; 
								\draw  (1,-.5) node[below]{\scalebox{1.5}{$\nexists$}}  ; 	
	\draw  (2.5,-2) node[above]{\scalebox{0.8}{$\mu_{1,c}^+$}} ; 	
	\draw  (3.3,-2) node[above]{\scalebox{0.8}{$\mu_{2,c}^+$}} ; 
	\draw  (4.3,-2) node[above]{\scalebox{0.8}{$\mu_{n,c}^+$}} ; 
	\draw [thick] (-.1,-1.4) node[left]{\scalebox{0.8}{$c^*$}} -- (.1,-1.4); 
	\draw [thick,dashed] (0,-1.4) -- (6,-1.4);
	
	\draw  (2.3,2.4) node[above]{\scalebox{0.8}{$\mu=\overline{\mu}$}} ;
	\end{tikzpicture}
	\caption{Energy curves for \eqref{quasi} with $q<p$. No solution exists for $c>0$ and $\mu<\mu_{1,c}^-$ and for $c \in (c^*,0)$ and $\mu<\mu_{1,c}^+$} \label{fig:CC}
\end{figure}
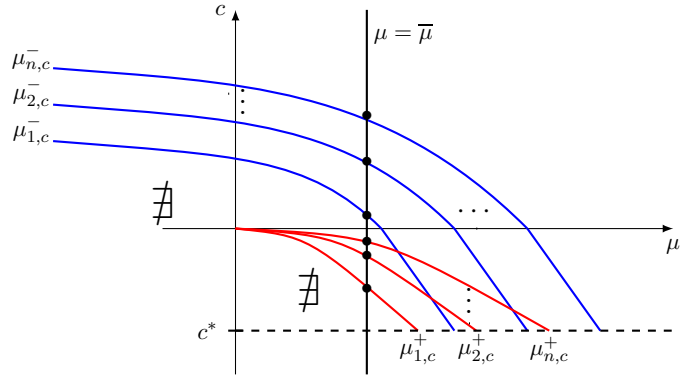

The bifurcation scenario provided by Theorem \ref{THMAP2} and Theorem \ref{EC2} extends some results known for the concave-convex in the semilinear case $p=2$. 
In \cite{BW} (see also \cite{ABC}) the authors proved that for any $\mu>0$  this problem has infinitely many solutions whose energy grows up to infinity as $n \to +\infty$, as well as infinitely many solutions  whose energy is negative and converges to zero as $n \to +\infty$.  A similar result holds for the $p$-Laplacian concave-convex problem with $\mu>0$ small and $r=p^*$, cf. \cite{GP}. Furthermore, the existence of infinitely many continua $\mathcal{C}_n$ of radial solutions bifurcating from $(0,0)$ has been observed for $\Omega$ being a ball \cite{AGP}, and for $p=2$ and $\Omega$ being an annulus \cite{BM}. We suspect that 
the curves $\mu_{n,c}^+$ and $\mu_{n,c}^-$ meet up at some point, so that their union correspond to the continuum $\mathcal{C}_n$ obtained in \cite{AGP,BM}. We refer to \cite{QSS} for  discussion on this issue.

\subsection{A Schr\"odinger-Bopp-Podolsky problem}

Finally we consider the energy functional associated to problem \eqref{SP}, namely,
$\Phi_\mu:H_r^1(\mathbb{R}^3)\to \mathbb{R}$  given by
\begin{equation*}
	\Phi_\mu(u)=\frac{1}{2}\int_{\mathbb{R}^3} |\nabla u|^2+\frac{\omega}{2}\int_{\mathbb{R}^3} |u|^2+\frac{\mu}{4}\int_{\mathbb{R}^3} \phi_uu^2-\frac{1}{p}\int_{\mathbb{R}^3} |u|^p.
\end{equation*}
 
 Recall that $p\in(2,3)$, $\omega>0$, and $a\ge 0$.
As in the previous problem, we shall obtain two sequences of solutions for some values of $c$: 

\begin{theorem}\label{THMAP3} Under the previous conditions there exists $c^*>0$ such that: 
		\begin{enumerate}
		\item For any $c<c^*$ there exist infinitely many $(\mu_{n,c}^-,u_{n,c})  \in (0,\infty) \times H_r^1(\mathbb{R}^3)\setminus\{0\}$ such that $\Phi_{\mu_{n,c}^-}(\pm u_{n,c})=c$ and $\Phi'_{\mu_{n,c}^-}(\pm u_{n,c})=0$, i.e., $\pm u_{n,c}$ are weak solutions of \eqref{SP} with $\mu=\mu_{n,c}^-$, having energy $c$, for every $n$.  Moreover:
		\begin{enumerate}
			\item  $\mu_{n,c}^-$ is non-increasing, $\displaystyle \lim_{n\to +\infty}\mu_{n,c}^-=0$ and $\|u_{n,c}\|\to +\infty$ as $n\to +\infty$, so $(0,+\infty)$ is a bifurcation point.
			\item If $c<0$ and $\mu>\mu_{1,c}^-$ then \eqref{SP} has no radial weak solution having energy $c$.\\
		\end{enumerate}
		
		\item For any $c\in (0,c^*)$  there exist infinitely many $(\mu_{n,c}^+,v_{n,c}) \in \mathbb{R} \times H_r^1(\mathbb{R}^3)\setminus\{0\}$ such that $\Phi_{\mu_{n,c}^+}(\pm v_{n,c})=c$ and $\Phi'_{\mu_{n,c}^+}(\pm v_{n,c})=0$, i.e., $\pm v_{n,c}$ are weak solutions of \eqref{SP} with $\mu=\mu_{n,c}^+$, having energy $c$, for every $n$. Moreover:
		\begin{enumerate}
			\item $\mu_{n,c}^+$ is non-increasing, $\displaystyle \lim_{n\to +\infty}\mu_{n,c}^+=-\infty$ and $v_{n,c} \rightharpoonup 0$ as $n\to +\infty$.
			\item $\mu_{n,c}^+<\mu_{n,c}^-$ for every $n$.
		\end{enumerate}
	\end{enumerate}
\end{theorem}	
	
\medskip
Now the maps $c \mapsto \mu_{n,c}^{\pm}$ enjoy the following behavior:
\medskip

\begin{theorem}\label{EC3}
	Under the conditions of Theorem \ref{THMAP3} the following properties hold for every $n\in \mathbb{N}$ (see Figure \ref{fig:SP}):
	\begin{enumerate}
		\item  The map $c \mapsto \mu_{n,c}^-$ is continuous and non-decreasing in $(-\infty,c^*)$, and $\displaystyle \lim_{c\to -\infty}\mu_{n,c}^-=0$. 
		\item The map $c \mapsto \mu_{n,c}^+$ is continuous and non-decreasing in $(0,c^*)$, and $\displaystyle \lim_{c\to 0^+}\mu_{n,c}^+=-\infty$. 
		\item For every $\mu\in (0,\mu_{n,0}^-)$ the problem \eqref{SP} has at least $n$ pairs of radial solutions with negative energy. 
		\item For every $\mu< \mu_{n,0}^+$ the problem \eqref{SP} has at least $n$ pairs of radial solutions with positive energy. 
	\end{enumerate}
\end{theorem}

\begin{figure}[h]
	\centering
	\begin{tikzpicture}[>=latex]
	\draw[->] (-1,0) -- (5,0) node[below] {\scalebox{0.8}{$\mu$}};
	\foreach \x in {}
	\draw[shift={(\x,0)}] (0pt,2pt) -- (0pt,-2pt) node[below] {\footnotesize $\x$};
	\draw[->] (0,-1.5) -- (0,2) node[left] {\scalebox{0.8}{$\mbox{Energy}$}};
	\foreach \y in {}
	\draw[shift={(0,\y)}] (2pt,0pt) -- (-2pt,0pt) node[left] {\footnotesize $\y$};
	\node[below left] at (0,0) {\footnotesize $0$};
	\draw [thick] (-.1,1.5) node[left]{\scalebox{0.8}{$c^*$}} -- (.1,1.5); 
	\draw [thick,dashed] (0,1.5) -- (6,1.5);
	
	\draw [thick] (0.4,-1.5) -- (0.4,2);
	
	\draw[red,thick] (0.3,-2) .. controls (0.4,-.7) .. (4,1.5);
	\draw[red,thick] (0.2,-2) .. controls (0.3,-.4) .. (3.3,1.5);
	\draw[red,thick] (0.1,-2) .. controls (0.2,-.1)  .. (2.1,1.5);
	\draw  (3,.7) node[below]{\scalebox{1.5}{$\nexists$}} ;
		\draw  (2,-.5) node[below]{\scalebox{1.5}{$\nexists$}}  ; 	
	\draw[red]  (4,1.5) node[above]{\scalebox{0.7}{$\mu_{1,c}^-$}} ; 	
	\draw[red]  (3.3,1.5) node[above]{\scalebox{0.7}{$\mu_{2,c}^-$}} ; 
	\draw[red]  (2.1,1.5) node[above]{\scalebox{0.7}{$\mu_{n,c}^-$}} ;

	\draw[blue,thick] (-1,.3) .. controls (2,0.55) .. (3.5,1.5);
	\draw[blue,thick] (-1,.6) .. controls (2,0.85) .. (3,1.5);
	\draw[blue,thick] (-1,.9) .. controls (1.8,1.2) .. (2,1.5);
	\draw[blue]  (-1.3,0) node[above]{\scalebox{0.7}{$\mu_{1,c}^+$}} ; 	
	\draw[blue]  (-1.3,.35) node[above]{\scalebox{0.7}{$\mu_{2,c}^+$}} ; 
	\draw[blue]  (-1.3,.7) node[above]{\scalebox{0.7}{$\mu_{n,c}^+$}} ; 
	\draw  (-.6,-.5) node[above]{\scalebox{1.5}{$\nexists$}} ;
	
	\draw (.4,-.2) node{\scalebox{0.6}{$\bullet$}};
	\draw (.4,-.75) node{\scalebox{0.6}{$\bullet$}};
	\draw (.4,-1.25) node{\scalebox{0.6}{$\bullet$}};
	\draw (.4,.4) node{\scalebox{0.6}{$\bullet$}};
	\draw (.4,.7) node{\scalebox{0.6}{$\bullet$}};
	\draw (.4,1.05) node{\scalebox{0.6}{$\bullet$}};
	\end{tikzpicture}
	\caption{Energy curves for \eqref{SP}} \label{fig:SP}
	
\end{figure}
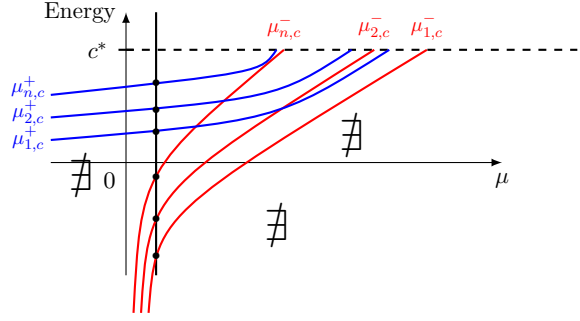

In the case $a=0$
this result should be compared with  
\cite[Theorem 3.1]{AR}, 
where the authors show the existence of multiple solutions with positive and negative energy
for small values of the parameter $\mu$. Moreover when $a>0$ Theorem \ref{EC3} improves \cite[Theorem 1.1]{DS} where the existence of a nontrivial 
solution at a positive energy level is proved for small values of $\mu$. We also refer to \cite{DMT} for the existence of infinitely many high energy solutions for a problem related to \eqref{SP}.

\medskip
\section{Proof of the abstract results}\label{sec:proofs}

The next  results are at the core of our approach. Throughout this section we assume that \eqref{H1} holds. In addition, $t(c,u)$ can be either $t^+(c,u)$ or  $t^-(c,u)$, i.e. a nondegenerate minimizer or a nondegenerate maximizer of $\psi_{c,u}$.
\begin{lemma}\label{l0} The following statements hold.
	\begin{enumerate}[label=(\roman{*}), ref=\roman{*}]
	\item\label{I0-i} The map $(c,u) \mapsto t(c,u)$ belongs to $C^1(\mathcal{I} \times X\setminus\{0\})$ and 
	\begin{equation}\label{dt}
	\frac{\partial t}{\partial c}(c,u)=\frac{-I_2'(t(c,u)u)u}{I_2(t(c,u)u)^2\psi''_{c,u}(t(c,u))} \qquad \forall (c,u) \in \mathcal{I} \times X\setminus\{0\}.
	\end{equation}
	In particular, for every $u \in X\setminus\{0\}$, the map
	$c \mapsto t(c,u)$ is increasing (respect. decreasing) if
	 $\psi''_{c,u}(t(c,u))I_2'(t(c,u)u)u<0$ (respect. $>0$); \medskip 
		\item\label{I0-ii}  $\Lambda \in C^1(\mathcal{I} \times X \setminus \{0\})$, and for any $v\in X$
		\begin{equation} \label{lp}
\frac{\partial \Lambda}{\partial u}(c,u)v=\frac{\Phi_{\Lambda(c,u)}'(t(c,u)u)t(c,u)v}{I_2(t(c,u)u)}, \qquad \forall (c,u) \in \mathcal{I} \times X\setminus\{0\}.
\end{equation}
In particular, $\frac{\partial \Lambda}{\partial u}(c,u)u=0$ for any $(c,u) \in \mathcal{I} \times X\setminus\{0\}$.
 Furthermore, 
\begin{equation} \label{dl}
\frac{\partial \Lambda}{\partial c}(c,u)=-\frac{1}{I_2(t(c,u)u)},
\end{equation}
so that for every $u \in X\setminus\{0\}$ the map $c \mapsto \Lambda(c,u)$ is decreasing (respect. increasing) if $I_2(t(c,u)u)>0$ (respect. $<0$). \medskip
	\item \label{I0-iii}For every $c>0$ the maps $u \mapsto t(c,u),\Lambda(c,u)$ are $(-1)$-homogeneous and $0$-homogeneous, respectively.
	
	\end{enumerate}
\end{lemma}

\begin{proof} 
\eqref{I0-i} Let $F:\mathcal{I} \times (0,\infty)\times X\setminus\{0\}\to \mathbb{R}$ be given by $F(c,t,u)=\psi'_{c,u}(t)$. Note that
	\begin{equation*}
		F(c,t(c,u),u)=0  \quad \mbox{and} \quad \frac{\partial F}{\partial t}(c,t(c,u),u)=\psi''_{c,u}(t(c,u))\neq0, \quad \forall (c,u) \in \mathcal{I} \times X\setminus\{0\}.
	\end{equation*}
	From the Implicit Function Theorem we deduce the first assertion. In addition, from $F(c,t(c,u),u)=0$ it follows that	\begin{equation*}
		\frac{\partial F}{\partial c}(c,t(c,u),u)+\frac{\partial F}{\partial t}(c,t(c,u),u)\frac{\partial t}{\partial c}(c,u)=0.
	\end{equation*}
Since $\frac{\partial F}{\partial c}(c,t,u)=\frac{I_2'(tu)u}{I_2(tu)^2}$, we infer \eqref{dt}.

\smallskip

\eqref{I0-ii}  Since  $(c,u) \mapsto \mu(c,u),t(c,u) \in C^1(\mathcal{I} \times X \setminus \{0\})$, 
		it follows from  \eqref{eq:Lambda} that  $\Lambda \in C^1(\mathcal{I} \times X \setminus \{0\})$. Computing we get
		 \begin{equation*} 
	\frac{\partial \Lambda}{\partial u}(c,u)v=\frac{\left[I'_1(t(c,u)u)-\Lambda(c,u)I'_2(t(c,u)u\right]\left(t(c,u)v+\left(\frac{\partial t}{\partial u}(c,u)v\right)u\right)}{I_2(t(c,u)u)},
\end{equation*}
and from
$$\left[I'_1(t(c,u)u)-\Lambda(c,u)I'_2(t(c,u)u)\right]u=I_2(t(c,u)u)\psi_{c,u}'(t(c,u))=0,$$
we obtain \eqref{lp}. Taking $v=u$ we find that
$$\frac{\partial \Lambda}{\partial u}(c,u)u=\frac{\Phi_{\Lambda(c,u)}'(t(c,u)u)t(c,u)u}{I_2(t(c,u)u)}=t(c,u)\psi_{c,u}'(t(c,u))=0.$$
In addition
 $$\frac{\partial \mu}{\partial u}(c,t(c,u)u)=\psi'_{c,u}(t(c,u))=0,$$
so that  by \eqref{eq:Lambda}  
 we deduce that $$\frac{\partial \Lambda}{\partial c}(c,u)=\frac{\partial \mu}{\partial c}(c,t(c,u)u) +\frac{\partial \mu}{\partial u}(c,t(c,u)u)\frac{\partial t}{\partial c}(c,u)u=\frac{\partial \mu}{\partial c}(c,t(c,u)u)=-\frac{1}{I_2(t(c,u)u)}.$$
	
\smallskip
	
	\eqref{I0-iii} Note that $\psi_{c,su}(t)=\mu(c,tsu)=\psi_{c,u}(st)$ for any $s,t>0$, so that $\psi'_{c,su}(t)=\psi_{c,u}'(st)s$, which implies by uniqueness that $t(c,su)=s^{-1}t(c,u)$. It follows that $\Lambda(c,su)=\Lambda(c,s^{-1}t(c,u) su)=\Lambda(c,u)$.
\end{proof}

From \eqref{lp} and the definition of $\Lambda(c,u)$ we derive the following result:

\begin{corollary}\label{c1}
	If $\frac{\partial \Lambda}{\partial u}(c,u)=0$ then $\Phi_{\Lambda(c,u)}'(t(c,u)u)=0$ and $\Phi_{\Lambda(c,u)}(t(c,u)u)=c$. 
\end{corollary}

Recall that since $\|\cdot \| \in C^1(X \setminus \{0\})$ the tangent space to $S$ at $u$
is given by  
$$\mathcal{T}_{u}(S)=\{v \in X: i'(u)v=0\}$$
where $i(u)=\frac{1}{2}\|u\|^2$, and then $X= \mathcal T_{u}S\oplus \mathbb Ru$.

An interesting feature of the functionals $\Lambda(c,\cdot)$ is the fact that the unit sphere $S$ is a
{\sl natural constraint} to find critical points of these functionals:\\

\begin{proposition}\label{propciritcalrestricted}  
Let $u \in S$. Then $\frac{\partial \widetilde{\Lambda}}{\partial u}(c,u)=0$ if and only if $\frac{\partial \Lambda}{\partial u}(c,u)=0$.
		\end{proposition}
		
\begin{proof} Let $u \in S $ be such that $\frac{\partial \widetilde{\Lambda}}{\partial u}(c,u)=0$. From Lemma 
\ref{l0} \eqref{I0-ii} we know that  $\frac{\partial \Lambda}{\partial u}(c,u)u=0$. Note that $\frac{\partial \widetilde{\Lambda}}{\partial u}(c,u)=\frac{\partial \Lambda}{\partial u}(c,u)_{| \mathcal{I} \times \mathcal{T}_{u}(S)}$.
If $w\in X$, then $w=v+su$ for some $v\in \mathcal{T}_{u}(S)$ and $s\in \mathbb{R}$, which implies that $\frac{\partial \Lambda}{\partial u}(c,u)w=\frac{\partial \widetilde{\Lambda}}{\partial u}(c,u)v+\frac{\partial \Lambda}{\partial u}(c,u)tu=0$. 
Since the converse is obvious, the proof is complete.
\end{proof}

\medskip
\subsection{Proof of Theorem \ref{THM1}} 
  \eqref{THM1-2}	
Let $u \in X \setminus \{0\}$ be such that  $\Phi'_\mu(u)=0$ and $\Phi_\mu(u)=c$. Then  $\mu=\frac{I_1(u)-c}{I_2(u)}=\psi_{c,u}(1)$. In addition, \eqref{mut} yields that $\psi_{c,u}'(1)=\frac{\Phi_{\psi_{c,u}(1)}'(u)u}{I_2(u)}=\frac{\Phi_{\mu}'(u)u}{I_2(u)}=0$,
i.e. $t(c,u)=1$. Thus $\mu=\Lambda(c,u) \geq \mu_{1,c}$ (respect. $\leq \mu_{1,c}$) if $u \mapsto \Lambda(c,u)$ is bounded from below (respect. above).

 
 	\smallskip


 \eqref{THM1-1} First we consider the case where \eqref{H2} holds with $u \mapsto \widetilde{\Lambda}(c,u)$ bounded from below.
By the Ljusternick-Schnirelman theorem (see e.g. \cite[Corollary 4.17]{G} or \cite{S}) there exist infinitely many $ u_{n,c} \in  S$ such that
$$\frac{\partial \Lambda}{\partial u}(c,\pm u_{n,c})=0\quad \text{and}\quad
\Lambda(c,\pm u_{n,c}) = \mu_{n,c} \quad \forall n\in \mathbb{N}$$
(we used here Proposition \ref{propciritcalrestricted}).
From  Corollary \ref{c1} and the fact that $u \mapsto t(c,u)$ is even, 
the sequence $v_{n}:=t(c,u_{n,c})u_{n,c}$
satisfies
$$\Phi_{\mu_{n,c}}( v_{n,c})= c\quad\text{and}\quad
\Phi'_{\mu_{n,c}}(v_{n,c}) = 0 \quad \forall n\in \mathbb N.$$
Now, if \eqref{H2} holds with $u \mapsto \widetilde{\Lambda}(c,u)$ bounded from above, then we deal with the functional $u \mapsto -\widetilde{\Lambda}(c,u)$, which is bounded from below. Moreover,  since $\frac{\partial \widetilde{\Lambda}}{\partial u}(c,u)=-(\frac{\partial (-\tilde{\Lambda)}}{\partial u}(c,u))$  we see that $u \mapsto -\widetilde{\Lambda}(c,u)$  satisfies the Palais-Smale condition at the level $\mu$ if and only if $u \mapsto \widetilde{\Lambda}(c,u)$ satisfies the Palais-Smale condition at the level $-\mu$.
Since  by assumption $u \mapsto \widetilde{\Lambda}(c,u)$ satisfies this condition at 
$$\mu_n=\displaystyle \sup_{F\in \mathcal{F}_n}\inf_{u\in F} \widetilde{\Lambda}(c,u)=-\displaystyle \inf_{F\in \mathcal{F}_n}\sup_{u\in F} (-\widetilde{\Lambda}(c,u)),$$
 we deduce that $-\mu_n$ is a critical value of $u \mapsto -\widetilde{\Lambda}(c,u)$, so that $\mu_n$ is a critical value of $u \mapsto \widetilde{\Lambda}(c,u)$ for every $n$, and conclude as in the previous case.

Finally, the second assertion follows from \cite[Lemma A.1]{QSS}. \qed

\subsection{Proof of Theorem \ref{THM2}}
Let us fix $n$. First of all, since $c \mapsto \Lambda(c,u)$ is decreasing in $\mathcal{I}$ it is clear that $c \mapsto \mu_{n,c}$ is nonincreasing in $\mathcal{I}$. Let us fix $[a,b]\subset \mathcal{I}$. We claim that there exist  $T>0$ such that for all $c \in [a,b]$ there exists a Palais-Smale sequence $\{u_{k,c}\}$ of the functional $u \mapsto \tilde{\Lambda}(c,u)$ at the level $\mu_{n,c}$ such that 
$\{u_{j,c}\} \subset S_{a,T}:=\{u \in S: \Lambda(a,u) \leq T \}$. Indeed, otherwise we could find a sequence $\{c_k\} \subset [a,b]$ such that for any $k$ there exists a Palais-Smale sequence $\{u_{j,c_k}\}$ of the functional $u \mapsto \tilde{\Lambda}(c_k,u)$ at the level $\mu_{n,c_k}$ with $\{u_{j,c_k}\} \not \subset S_{a,T}$. Thus we can extract a sequence $\{u_k\} \subset S$ such that $\Lambda(c_k,u_k)$ is bounded, $\frac{\partial \tilde{\Lambda}}{\partial u}(c_k,u_k) \to 0$
and  $\Lambda(a,u_k) \to+ \infty$. Since $c_k \to c$ we infer by \eqref{PSG} that $\{u_k\}$ is compact and hence the set $K=\{(a,u_k):k\in \mathbb{N}\}$ is compact, which contradicts \eqref{H3}. Therefore
\begin{equation*}
\mu_{n,c}=\inf_{F\in \widetilde{\mathcal{F}}_n}\sup_{u\in F}\Lambda(c,u),
\end{equation*}
where $\widetilde{\mathcal{F}}_n=\{F\in \mathcal{F}_n: F \subset S_{a,T} \}$. 
Now, by the mean value theorem we have
$|\Lambda(c_2,u)-\Lambda(c_1,u)|=\frac{\partial \Lambda}{\partial c}(\theta,u)|c_2-c_1|$,
where $\theta:=\theta(u)\in (\min\{c_1,c_2\},\max\{c_1,c_2\} )$. Since $c \mapsto \Lambda(c,u)$ is decreasing, we see that $\Lambda$ is bounded in $[a,b] \times S_{a,T}$, and
by \eqref{H3}  there exist $M,m>0$ such that $m|c_2-c_1|\leq |\Lambda(c_2,u)-\Lambda(c_1,u)|\leq M|c_2-c_1|$ for $c_1,c_2 \in [a,b]$ and $u \in S_T$. It follows that $m|c_2-c_1|\leq |\mu_{n,c_1}-\mu_{n,c_2}| \leq M|c_2-c_1|$ for $c_1,c_2 \in [a,b]$, so that $c \mapsto \mu_{n,c}$ is decreasing and locally  Lipschitz continuous. The proof in the case that $c \mapsto \Lambda(c,u)$ is increasing in $\mathcal{I}$ is similar, so we omit it. \qed

\medskip

In the sequel we consider three classes of functionals, which are inspired by (but not limited to) problems \eqref{quasi} with $q=p$, $q<p$, and \eqref{SP}, respectively. The next sections have the same structure, namely:
\begin{itemize}
	\item general assumptions.
	\item expressions of $\mu(c,u), \psi_{c,u}$ and verification of  \eqref{H1};
	\item analysis of the maps $t^\pm, \Lambda^\pm$ and verification of \eqref{H2};
	\item application of Theorems \ref{THM0} and \ref{THM1};
	\item analysis of the maps $c \mapsto \mu_{n,c}$: continuity, monotonicity and asymptotic behaviour;
\item main results of the section.
\end{itemize}
\medskip

\section{A first class of functionals}
\label{sec:4}
\medskip

We consider the functional
\begin{equation*}
	\Phi_\mu(u):=\frac{1}{\eta}\left(N(u) -\mu A(u)\right)-\frac{1}{\beta}B(u)
\end{equation*}
where  $1<\eta,\beta$, and $N,A,B \in C^1(X)$ are even functionals satisfying the following conditions:
\begin{enumerate}[label=(\arabic*),ref=\arabic*]
	\item\label{st-1} $N,A$ are  $\eta$-homogeneous, and $B$ is $\beta$-homogeneous.
	\item\label{st-2} $A(u)>0$ and $B(u)\neq 0$ for any $u \neq 0$.
	\item\label{st-3} There exists $C>0$ such that $ N(u)\geq C^{-1}\|u\|^{\eta}$, $A(u) \leq C\|u\|^\eta$ and $|B(u)|\leq C\|u\|^{\beta}$ for all $u \in X$. 
	\item\label{st-4} $A'$ and $B'$ are completely continuous, i.e. $A'(u_n) \to A'(u)$ and $B'(u_n) \to B'(u)$ in $X^*$ if $u_n \rightharpoonup u$ in $X$.
	\item\label{st-5} $N$ is weakly lower semicontinuous and there exists $C>0$ such that 
	$$(N'(u)-N'(v))(u-v)\geq C(\|u\|^{\eta-1}-\|v\|^{\eta-1})(\|u\|-\|v\|) \quad \forall u,v \in X.$$ 
\end{enumerate}

In all the results in this section the above conditions are implicitly assumed.

Let $c \in \mathbb{R}$ be fixed. For $u\in X\setminus\{0\}$, we have 
\begin{equation*}
	\mu(c,u):=\displaystyle \frac{\eta}{A(u)}\left(\frac{1}{\eta}N(u)-\frac{1}{\beta}B(u)-c\right).
\end{equation*}
Thus
\begin{equation}\label{st-psi}
	\psi_{c,u}(t)=\frac{N(u)}{A(u)}- \frac{\eta B(u)}{\beta A(u)}t^{\beta-\eta}-\frac{\eta c}{A(u)}t^{-\eta},\quad t>0
\end{equation}
which is cleary $C^{2}$ in $(0,\infty)$. Moreover,
\begin{equation}\label{0dp}
	\psi_{c,u}'(t)=- (\beta-\eta)\frac{\eta B(u)}{\beta A(u)}t^{\beta-\eta-1}+\frac{\eta^2 c}{A(u)}t^{-\eta-1},
\end{equation}
so that the map $(c,u,t)\mapsto \psi_{c,u}'(t)$ is clearly $C^{1}$ in $\mathbb{R} \times X \setminus \{0\} \times (0,\infty)$.

Note that if $u$ is a critical point of $\Phi_\mu$ then the homogeinity of $N,A$ and $B$ yields that
$N(u) -\mu A(u)=B(u)$, and then
\begin{equation}\label{epm}
\Phi_{\mu}(u)=\left(\frac{1}{\eta}-\frac{1}{\beta}\right)\left(N(u) -\mu A(u)\right)=\frac{\beta-\eta}{\beta \eta}B(u).
\end{equation}
It follows that if $\Phi_{\mu}(u)=c$ and $\eta<\beta$ then $B(u)$ and $c$ have the same sign, whereas for $\eta>\beta$ they have opposite signs.
\medskip
\subsection{The case $\eta<\beta$}
In this subsection we deal with the superhomogeneous case $\eta<\beta$.
From \eqref{epm},  the following alternative must occur for $u \neq 0$ solving \eqref{ef}:
$c>0 \mbox{ and } B(u)>0$ or  $c<0 \mbox{ and } B(u)<0$.
Therefore  we shall assume throughout this section one of the following conditions:
\begin{enumerate}[label=(C\arabic*),ref=C\arabic*,start=1]
	\item \label{C1} $B(u)>0$ for any $u \neq 0$. 
	\item \label{C2}   $B(u)<0$ for any $u \neq 0$.
\end{enumerate}
\medskip

We start proving that \eqref{H1} is satisfied. The next lemma follows promptly from \eqref{st-psi} and \eqref{0dp}:
\begin{lemma}\label{0lem:H1}
	For any $u \in X \setminus \{0\}$ and $c > 0$ (respect. $c<0$) the map $\psi_{c,u}$ has a unique critical point $t^-(c,u)>0$ (respect. $t^+(c,u)>0$),
	which is a non-degenerate global maximizer (respect. minimizer) if \eqref{C1} (respect.  \eqref{C2}) holds. Moreover
	\begin{equation}\label{tea}
	t^{\pm}(c,u)=\left(\frac{\eta \beta}{\beta-\eta} \frac{c}{B(u)}\right)^{\frac{1}{\beta}}=\left(\frac{\eta \beta}{\beta-\eta}\right)^{\frac{1}{\beta}}|B(u)|^{-\frac{1}{\beta}}|c|^{\frac{1}{\beta}},
	\end{equation}
and the map $c \mapsto t^{-}(c,u)$ (respect. $t^{+}(c,u)$) 
is increasing in $(0,+\infty)$ (respect. decreasing in $(-\infty,0)$) if \eqref{C1} (respect. \eqref{C2}) holds. In particular, for any $c_1>0$ there exists $C>0$ such that $t^{\pm}(c,u) \geq C$
for any $u \in S$ and $|c| \geq c_1>0$.
\end{lemma}

Consequently  Lemma \ref{l0} \eqref{I0-i} provides that  $t^- \in C^{1}((0,\infty) \times (X \setminus \{0\}))$ and $t^+ \in C^{1}((-\infty,0) \times (X \setminus \{0\}))$.
 Next we prove some properties of the maps $\Lambda^\pm$. Let $D_c:=\left(\frac{\beta \eta |c|}{|\beta-\eta|}\right)^{\frac{\beta-\eta}{\beta}}$ for $c \in \mathbb{R}$, $\mathcal{I}_+=(0,\infty)$, and $\mathcal{I}_{-}=(-\infty,0)$.
 
  Then
\begin{equation}\label{eta=alphaex}
	\Lambda^{\pm}(c,u)=\frac{N(u)\pm D_c|B(u)|^{\frac{\eta}{\beta}}}{A(u)}, \quad \forall c \in \mathcal{I}_{\mp}, u \in S,
\end{equation}

Moreover, we denote by $\Lambda$ the continuous extension of $\Lambda^{\pm}$ to $c=0$, i.e.
\begin{equation}\label{eta=alphaex0}
	\Lambda(u)=\Lambda^{\pm}(0,u):=\frac{N(u)}{A(u)}, \quad \forall u\in X\setminus\{0\},
\end{equation}
and set $t^{\pm}(0,u)\equiv 0$.

The following result, which is a direct consequence of \eqref{tea}, \eqref{eta=alphaex}, and \eqref{eta=alphaex0}, shows that if $u \mapsto \Lambda^{\pm}(c_0,u)$  is bounded for some $c_0 \in \mathcal{I}_{\mp}$ then it is bounded for any $c\in \mathcal{I}_{\mp}$:

\begin{lemma}\label{ls0} Let  $S_0\subset S$ and $K\subset [0,\infty)$ (respect. $(-\infty,0]$) be a compact set if \eqref{C1} (respect. \eqref{C2}) holds. Then	
$$\sup_{K\times S_0}\Lambda^{\pm}<+\infty \Longleftrightarrow \sup_{S_0}N<+\infty, \quad \inf_{S_0}A>0 \quad \mbox{and}\quad \inf_{S_0}B>0. $$
	Moreover, in such a case we also have $\displaystyle \sup_{ K\times S_0}t^{\pm}<+\infty$, so that $t^{\pm}$ is bounded and away from zero in $K\times S_0$.	In particular, $\Lambda^{\pm}(c,u_n) \to +\infty$ for any $c\in \mathcal{I}_{\mp}$ if $u_n \rightharpoonup 0$ in $X$.
\end{lemma}

We shall use Lemma \ref{ls0} in several occasions.

\begin{lemma} \label{0pl} The following statements hold.
	\begin{enumerate}[label=(\roman*),ref=\roman*]
		\item\label{0pl-i} For every $c>0$ (respect. $c<0$)  the map $u \mapsto \Lambda^-(c,u)$ (respect. $u \mapsto \Lambda^+(c,u)$) is bounded from below on $S$ if \eqref{C1} (respect. \eqref{C2}) holds.
		\item\label{0pl-ii} For every $u\in X \setminus\{0\}$ the map $c\mapsto \Lambda^-(c,u)$ (respect. $c \mapsto \Lambda^+(c,u)$) is decreasing in $[0,\infty)$ (respect. in $(-\infty,0]$) if \eqref{C1} (respect. \eqref{C2}) holds.
		\item\label{0pl-iii} Suppose that \eqref{C1} holds and $\Lambda$ is bounded in  $S_0\subset S$.
		Then $\displaystyle \lim_{c \to +\infty}\Lambda^-(c,u)= -\infty$ uniformly in $S_0$.
		\item\label{0pl-iv} Suppose that \eqref{C2} holds. Then $\displaystyle \lim_{c \to -\infty}\Lambda^+(c,u)= +\infty$ uniformly in $S$.
	\end{enumerate}
\end{lemma}
\begin{proof} \eqref{0pl-i} and \eqref{0pl-ii} are clear from  \eqref{eta=alphaex}, \eqref{eta=alphaex0} and condition \eqref{st-3}.
	Let us prove  \eqref{0pl-iii}. by Lemma \ref{ls0} there exist $C_1,C_2>0$ such that
	$\Lambda^-(c,u) \leq C_1-C_2 D_c$ for any $u \in S_0$ and $c>0$.
	Since $D_c \to +\infty$ as $c \to +\infty$, we obtain the desired conclusion.
	To prove	
	\eqref{0pl-iv}, assume by contradiction that 
	 $c_n\to -\infty$ and $\{u_n\}\subset S$ are such that $$M \ge \Lambda^+(c_n,u_n)=\frac{N(u_n)+ D_c|B(u_n)|^{\frac{\eta}{\beta}}}{A(u_n)} \ge C_1D_c |B(u_n)|^{\frac{\eta}{\beta}}.$$ It follows that $B(u_n)\to 0$ and, by condition \eqref{st-4}, we obtain $u_n \rightharpoonup 0$, so that $$\Lambda^+(c_n,u_n)\ge \frac{N(u_n)}{A(u_n)} \to+ \infty,$$ a contradiction. Thus $ \lim_{c\to -\infty}\Lambda^+(c,u)=+\infty$ uniformly with respect to $u\in S$, and the proof is complete.
\end{proof}
The next result implies that \eqref{H2} is satisfied.

\begin{proposition}\label{0PS} 
	Suppose that $c_n\to c\in \mathcal{I}_{\mp}$
	and $\{u_{n}\}\subset S$ are such that $\frac{\partial \widetilde{\Lambda^\pm}}{\partial u}(c_n,u_n)\to 0$ and $\{\widetilde{\Lambda^\pm}(c_n,u_n)\}$ is bounded. Then 
	$\{u_n\}$ has  a convergent subsequence. In particular, the maps $u \mapsto \widetilde{\Lambda^\pm}(c,u)$ satisfy the Palais-Smale condition for any $c\in \mathcal{I}_{\mp}$.
\end{proposition}
\begin{proof} We deal only with $c>0$ and $\widetilde{\Lambda^-}$, since the other case is similar.  By Lemma \ref{0lem:H1} and Lemma \ref{ls0} we know that $t^-(c_n,u_n)$ and $A(u_n)$ are  bounded and away from zero. Therefore, up to a subsequence, the assumptions of Lemma \ref{psg} are satisfied (see also Remark \ref{rps}), and we reach the desired conclusion.
\end{proof}

By the results above and Theorem \ref{THM1}, 
 the sequences

$$ \mu_{n,c}^- = \inf_{F\in \mathcal F_{n}} \sup_{u\in F}\widetilde{\Lambda^-}(c,u), \quad \mbox{if }  \eqref{C1} \mbox{ holds and } c>0,$$
and
$$\mu_{n,c}^+ = \inf_{F\in \mathcal F_{n}} \sup_{u\in F}\widetilde{\Lambda^+}(c,u),\quad \mbox{if }  \eqref{C2} \mbox{ holds and }c<0,$$ 
give rise to a sequence  $\{u_{n,c}\}$ of critical points of the families $\Phi_{\mu_{n,c}^{-}}$ and $\Phi_{\mu_{n,c}^{+}}$ having energy equal to $c$.


\begin{lemma} \label{lii}
Let $c \neq 0$. Then $\mu_{n,c} \to +\infty$ as $n \to +\infty$.  If, in addition  $N$ is bounded on bounded sets of $X$, then  $\|u_{n,c}\| \to+ \infty$ as $n \to +\infty$.
\end{lemma}

\begin{proof}
Let us fix $c$ and write $\mu_n= \mu_{n,c}$ and $u_n=u_{n,c}$. 
The first assertion follows from Theorem \ref{THM1} and Lemma \ref{ls0}. Let us assume that $N$ is bounded on bounded sets of $X$. Assume by contradiction that $\{u_n\}$ is bounded in $X$, then we can assume that $u_n \rightharpoonup u$ in $X$. From \eqref{epm} and $\Phi_{\mu_n}(u_n)=c$ we find that $B(u_n)=\frac{\beta \eta}{\beta-\eta}c$, so that $B(u)\neq 0$, i.e. $u \neq 0$. On the other hand, $\Phi_{\mu_n}'(u_n)=0$ implies that $A'(u_n)=\mu_n^{-1} (N'(u_n)-\frac{\eta}{\beta}B'(u_n))$. By homogeinity we infer that $A(u_n)\to 0$, so $u=0$, and we reach a contradiction.
\end{proof}

Let us proceed with the study of $\mu_{n,c}^\pm$ as functions of $c$. In view of \eqref{eta=alphaex0}, we set
\begin{equation}\label{vpp}
\mu_{n,0}^\pm:=\mu_{n}=\inf_{F\in \mathcal{F}_n}\sup_{u\in F}\frac{N(u)}{A(u)}.
\end{equation}

\begin{lemma}\label{0lip} 
For every $n \in \mathbb{N}$ the following properties hold.
	\begin{enumerate}[label=(\roman*),ref=\roman*]
		\item\label{0lip-0} 
		The map $c \mapsto \mu_{n,c}^\pm$ is decreasing and continuous in $\overline{\mathcal{I}_{\mp}}$, and locally Lipschitz continuous in $\mathcal{I}_{\mp}$.
		
		\item\label{0lip-i} $\displaystyle \lim_{c\to \mp \infty}\mu_{n,c}^{\pm}=\pm \infty$.
		
	\end{enumerate}
\end{lemma}

\begin{proof} 
\eqref{0lip-0} Indeed, by Lemma \ref{ls0} and \eqref{dl} we see that  \eqref{H3} is satisfied by $\Lambda^{\pm}$. Proposition \ref{0PS} shows that $\Lambda^{\pm}$ satisfy \eqref{PSG}, so Theorem \ref{THM2} yields that $c \mapsto \mu_{n,c}^\pm$ is decreasing and  locally Lipschitz continuous in $\mathcal{I}_{\mp}$. Moreover, arguing as in the proof of Theorem \ref{THM2} we see that $c \mapsto \mu_{n,c}^\pm$ is also continuous at $c=0$.

	\smallskip

\eqref{0lip-i} To prove that $\displaystyle \lim_{c\to +\infty}\mu_{n,c}^-=- \infty$ it suffices to pick some $F \in \mathcal{F}_n$ such that $\Lambda$ is bounded in $F$. By  Lemma \ref{0pl} \eqref{0pl-iii} we have that $\displaystyle \lim_{c \to +\infty}\Lambda^-(c,u)= -\infty$  uniformly in $F$, which yields the desired conclusion. On the other hand,  Lemma \ref{0pl} \eqref{0pl-iii} yields that $\mu_{n,c}^+= \Lambda^+(c,u_{n,c})\to +\infty$ as $c \to -\infty$.
\end{proof}

Next we establish some bifurcation type results, which imply, in particular, Theorem \ref{EC1} (1):

\begin{proposition}\label{bif}
	Assume that
	either \eqref{C1}  or \eqref{C2} hold. 
	\begin{enumerate}[label=(\roman*),ref=\roman*]
		\item \label{bif-1}	If $\{u_k\} \subset X$ and $\{\mu_k\} \subset \mathbb{R}$ are such that $\Phi_{\mu_k}'(u_k)=0$ and $|\Phi_{\mu_k}(u_k)| \to +\infty$  then $ \|u_k\| \to +\infty$. 
		\item \label{bif-2} If $\{u_k\} \subset X$ and $\{\mu_k\} \subset \mathbb{R}$ are such that $\{\mu_k\}$ is bounded from above, $\Phi_{\mu_k}'(u_k)=0$ and $\Phi_{\mu_k}(u_k) \to 0$  then $u_k \to 0$ in $X$.
	\end{enumerate}
	
\end{proposition}

\begin{proof}
Recall that by \eqref{epm} we have $\frac{\beta-\eta}{\beta \eta}B(u_k)=\Phi_{\mu_k}(u_k)$.
	 Then  $|B(u_k)| \to +\infty$  if $|\Phi_{\mu_k}(u_k)| \to +\infty$ and it follows,
	by \eqref{1st-3}, that $\|u_k\| \to+ \infty$, proving \eqref{bif-1}.
	
	Again by \eqref{epm},  if $\Phi_{\mu_k}(u_k)\to 0$ then $B(u_k) \to 0$. 
	If $\mu_k \to -\infty$ then $N(u_k)<B(u_k) \to 0$, so $u_k \to 0$ in $X$. 
	On the other hand, if $\{\mu_k\}$ is bounded  then one can show that $\{u_k\}$ is bounded, and from $B(u_k) \to 0$ we infer that $u_k \rightharpoonup 0$. Finally, $N(u_k)=\mu_k A(u_k) +B(u_k)$ yields that $u_k \to 0$ in $X$
	and \eqref{bif-2} is proved.
\end{proof}

\begin{corollary}\label{cbif}
	The following statements hold.
	\begin{enumerate}[label=(\roman*),ref=\roman*]
		\item\label{cbif-1} If \eqref{C1} holds then for every $\mu\in\mathbb{R}$ there exist two increasing sequences $\{c_{k}\} \subset  (0,+\infty)$ and $\{n_k\} \subset \mathbb{N}$ such that $\mu_{n_k,c_{k}}^-=\mu$ for every $k\in \mathbb{N}$, and $c_{k},n_k,\|u_{n_k,c_k}\| \to+ \infty$ as $k \to+ \infty$.
		\item\label{cbif-2} If \eqref{C2} holds then for every $\mu\in\mathbb{R}$ there exist two increasing sequences $\{-c_{k}\} \subset  (0,\infty)$ and $\{n_k\} \subset \mathbb{N}$ such that $\mu_{n_k,c_{k}}^+=\mu$ for every $k\in \mathbb{N}$, and $-c_{k},n_k,\|u_{n_k,c_k}\| \to+ \infty$ as $k \to+ \infty$.
	\end{enumerate}
\end{corollary}

\begin{proof}
	Let us assume that \eqref{C1} holds. Set $d_1=1$. We know from Lemma \ref{lii} that there exists $n_1\in \mathbb{N}$ such that $\mu_{n_1,d_1}^->\mu$. By Lemma \ref{0lip}  there exists $c_1>d_1$ such that $\mu_{n_1,c_1}^-=\mu$. Now take $d_2=c_1+2$, so that again by Lemma \ref{lii} we can find $n_2>n_1$ such that $\mu_{n_2,d_2}^->\mu$ and thus, by Lemma \ref{lii} and Lemma \ref{0lip}, there exists $c_2>d_2$ such that $\mu_{n_2,c_2}^-=\mu$. By repeating this procedure the existence of $\{c_{k}\}$ and $\{n_k\}$ is proved. From Proposition \ref{bif} we have $\|u_{n_k,c_k}\|\to+ \infty$. The proof of the second item is similar.
\end{proof}

\subsection{The case $\eta>\beta$}\label{sub}
The subhomogeneous case $\eta>\beta$ can be handled in a similar way. Let us briefly describe the main points. 

Now \eqref{epm} shows that if $\Phi_{\mu}(u)=c$ then $B(u)$ and $c$ have opposite signs. Hence, if $u\neq 0$ then one of the following two cases must occur:
$c<0<B(u) \quad \mbox{or} \quad c>0>B(u)$.
Let us focus on the first case and assume that  $B(u)>0$ for every $u \neq 0$. Then for any $c<0$ and $u \neq 0$ the map $\psi_{c,u}$ has a unique critical point $t^+(c,u)>0$, still given by \eqref{tea}, which is a non-degenerate global minimizer. Moreover, the map $c \mapsto t^+(c,u)$ is 
decreasing in $(-\infty,0)$ and uniformly away from zero for $c$ away from zero. The map $\Lambda^+$, given by $\Lambda^+(c,u)=\mu(c,t^+(c,u)u)$ for $c<0$ and $u \neq 0$, still satisfies \eqref{eta=alphaex}. It follows that Lemma \ref{ls0} remains valid. Moreover:
\begin{itemize}
	\item for any $c<0$
	the map $u \mapsto \Lambda^+(c,u)$ is bounded from below on $S$,
	\item for any $u \neq 0$ the map $c\mapsto \Lambda^+(c,u)$ is decreasing in $(-\infty,0)$,
	\item if the functional $\Lambda$ is bounded in $S_0\subset S$ then $\Lambda^+(c,u) \to -\infty$ uniformly in $S_0$ as $c \to 0^-$. 
\end{itemize} 
Furthermore, the proof of Proposition \ref{0PS} still applies, and we deduce that for any $c<0$  the map $u \mapsto \widetilde{\Lambda^+}(c,u)$ satisfies the Palais-Smale condition. Thus 
$ \mu_{n,c}^+ = \displaystyle \inf_{F\in \mathcal F_{n}} \sup_{u\in F}\widetilde{\Lambda^+}(c,u)$ are critical values of the functional $u \mapsto \Lambda^+(c,u)$.  Arguing as in the proof of Lemma \ref{0lip} we see that, for every $n$:
\begin{itemize}
	\item  the map $c \mapsto \mu_{n,c}^+$ is decreasing and locally Lipschitz continuous in $(-\infty,0)$,
	\item $\displaystyle \lim_{c\to 0^-}\mu_{n,c}^+=-\infty$, and $\displaystyle \lim_{c\to -\infty}\mu_{n,c}^+=\mu_{n}$, with $\mu_n$ given by \eqref{vpp}. 
\end{itemize}
One may also check that Proposition \ref{bif} remains valid, implying in particular that $u_{n,c} \to 0$ as $c \to 0^-$, i.e. $(-\infty,0)$ is a bifurcation point, and $\|u_{n,c}\| \to +\infty$ as $c \to -\infty$, i.e. every $(\mu_n,+\infty)$ is a bifurcation point.

Finally, Corollary \ref{cbif} read now as follows:
for every $\mu\in\mathbb{R}$ there exist two increasing sequences $\{c_{k}\} \subset  (-\infty,0)$ and $\{n_k\} \subset \mathbb{N}$ such that $\mu_{n_k,c_{k}}^-=\mu$ for every $k\in \mathbb{N}$, and $n_k\to+ \infty$,  $c_{k}\to 0^-$, and $u_{n_k,c_k}\to 0$  in  $X$ as $k \to +\infty$. Thus every $(\mu,0)$ is a bifurcation point.

We are now in position to state the following result for problem \eqref{ef}:

 \begin{theorem} \label{ts1} 
 Assume that  \eqref{C1} hold.
 		\begin{enumerate}
 	\item If $\eta<\beta$ (respect. $\eta>\beta$) then for every $c>0$ (respect. $c<0$) there exists a sequence $\{(\mu_{n,c},u_{n,c})\}  \subset \mathbb{R} \times X\setminus\{0\}$ such that $(\mu_{n,c},\pm u_{n,c})$ solves \eqref{ef} for every $n \in \mathbb{N}$.
 Moreover:
 \begin{enumerate}
 		\item $\{\mu_{n,c} \}$ is non-decreasing and  $\mu_{n,c} \to +\infty$ as $n\to +\infty$.
 		\item \eqref{ef} has no solution for $\mu<\mu_{1,c}$.
 	\end{enumerate} 
 \item If $\eta<\beta$ then:
 	\begin{enumerate}
 		\item For every $n\in \mathbb{N}$ the map $c \mapsto \mu_{n,c}$ is locally Lipschitz continuous and decreasing in $(0,+\infty)$, and continuous in $[0,+\infty)$.
 		\item  $\mu_{n,c}\to -\infty$ and $\|u_{n,c}\| \to +\infty$ as $c\to+ \infty$, i.e. $(-\infty,+\infty)$ is a bifurcation point.
 		\item   $\mu_{n,c}\to \mu_n$ and $u_{n,c} \to 0$ in $W_0^{1,p}(\Omega)$ as $c \to 0^+$, i.e. every $(\mu_n,0)$ is a bifurcation point.
 		\item For every $\mu\in \mathbb{R}$ there are infinitely many  $u_n\in X$ such that 
 		$\Phi_\mu'(\pm u_n)=0$, and $\Phi_\mu(u_n)\to +\infty$ and $\|u_n\|\to +\infty$ as $n \to \infty$, so $(\mu,+\infty)$ is a bifurcation point for any  $\mu\in \mathbb{R}$.
 	\end{enumerate}
 	\item If $\eta>\beta$ then:
 	\begin{enumerate}
 		\item For every $n\in \mathbb{N}$ the map $c \mapsto \mu_{n,c}$ is locally Lipschitz continuous and decreasing in $(-\infty,0)$.
 		\item $\mu_{n,c}\to -\infty$ and $u_{n,c} \to 0$ in $X$ as $c\to 0^-$, i.e. $(-\infty,0)$ is a bifurcation point.
 		\item $\mu_{n,c}\to \mu_n$ and $\|u_{n,c}\| \to +\infty$ as $c \to -\infty$, i.e. every $(\mu_n,+\infty)$ is a bifurcation point.
 		\item For every $\mu\in \mathbb{R}$ there are infinitely many  $v_n\in X$ such that 
 		$\Phi_\mu'(\pm v_n)=0$ and $\Phi_\mu(v_n)\to 0^-$ and $v_n\to 0$ in $W_0^{1,p}(\Omega)$  as $n \to +\infty$, so $(\mu,0)$ is a bifurcation point for any $\mu \in \mathbb{R}$.
 	\end{enumerate}
 \end{enumerate}
 \end{theorem}
 
\begin{proof}
The existence of the sequence $\{(\mu_{n,c},u_{n,c})\}  \subset \mathbb{R} \times X\setminus\{0\}$ follows from Lemma \ref{0pl}, Proposition \ref{0PS}, and Lemma \ref{lii}, combined with Theorem \ref{THM0} and Theorem \ref{THM1}; whereas the properties of the map $c \mapsto \mu_{n,c}$ are consequences of Lemma \ref{0lip}, Proposition \ref{bif} and Corollary \ref{cbif}. Note that the corresponding results in the case  $\eta>\beta$ were 
just described above.
\end{proof}

\subsection{Proof of Theorems \ref{THMAP1} and \ref{EC1}}
It is clear that $N(u)=\int_\Omega |\nabla u|^p$, $A(u)=\int_\Omega |u|^p$, and $B(u)=\int_\Omega |u|^r$ satisfy the conditions
\eqref{st-1}-\eqref{st-5} of this section with $\eta=p$ and $\beta=r$,
so Theorem \ref{ts1}  yield all the conclusions, except for:
\begin{itemize}
	\item $\|u_{n,c}\| \to +\infty$ as $n \to +\infty$, which follows from Lemma \ref{lii};
\item  $u_{n,c}$ is sign-changing for $n$ large enough, which is a consequence of the fact that \eqref{quasi} has no positive solution for $\mu>\mu_1$,  the first eigenvalue of the Dirichlet $p$-Laplacian, cf. \cite[Theorem 2.1]{AH}.
\end{itemize}

\medskip 
\section{A second class of functionals}
\label{SubCC}

We consider now the functional
\begin{equation*}
\Phi_\mu(u):=\frac{1}{\eta}N(u) -\frac{\mu}{\alpha}A(u)-\frac{1}{\beta}B(u)
\end{equation*}
where  $1<\alpha< \eta<\beta$, and $N,A,B \in C^1(X)$ are even functionals satisfying the following conditions:
\begin{enumerate}[label=(\arabic*),ref=\arabic*]
\item\label{1st-1} $N,A,B$ are  $\eta$-homogeneous, $\alpha$-homogeneous and $\beta$-homogeneous, respectively.
\item\label{1st-2} $A(u),B(u)>0$ for any $u \neq 0$.

\item\label{1st-3} There exists $C>0$ such that $ N(u)\geq C^{-1}\|u\|^{\eta}$, $A(u) \leq C\|u\|^\alpha$ and $B(u)\leq C\|u\|^{\beta}$ for all $u \in X$. 
\item\label{1st-4} $A'$ and $B'$ are completely continuous, i.e. $A'(u_n) \to A'(u)$ and $B'(u_n) \to B'(u)$ in $X^*$ if $u_n \rightharpoonup u$ in $X$.
\item\label{1st-5} $N$ is weakly lower semicontinuous and there exists $C>0$ such that 
	$$(N'(u)-N'(v))(u-v)\geq C(\|u\|^{\eta-1}-\|v\|^{\eta-1})(\|u\|-\|v\|) \quad \forall u,v \in X.$$ 
\end{enumerate}

The above conditions are implicitly assumed in all this section.

Let $c \in \mathbb{R}$ be fixed. For $u\in X\setminus\{0\}$, we have explicitly
\begin{equation*}
	\mu(c,u):=\displaystyle \frac{\alpha}{A(u)}\left(\frac{1}{\eta}N(u)-\frac{1}{\beta}B(u)-c\right).
\end{equation*}
Then
\begin{equation}\label{1st-psi}
	\psi_{c,u}(t)=\frac{\alpha N(u)}{\eta A(u)}t^{\eta-\alpha}- \frac{\alpha B(u)}{\beta A(u)}t^{\beta-\alpha}-\frac{\alpha c}{A(u)}t^{-\alpha},\quad t>0
\end{equation}
which is clearly a $C^{2}$ map in $(0,\infty)$, its first and second derivatives being given by
\begin{equation*}
\psi_{c,u}'(t)=(\eta-\alpha)\frac{\alpha N(u)}{\eta A(u)}t^{\eta-\alpha-1}- (\beta-\alpha)\frac{\alpha B(u)}{\beta A(u)}t^{\beta-\alpha-1}+\frac{\alpha^2 c}{A(u)}t^{-\alpha-1}.
\end{equation*}
and
\begin{equation*}
	\psi_{c,u}''(t)=(\eta-\alpha)(\eta-\alpha-1)\frac{\alpha N(u)}{\eta A(u)}t^{\eta-\alpha-2}- (\beta-\alpha)(\beta-\alpha-1)\frac{\alpha B(u)}{\beta A(u)}t^{\beta-\alpha-2}-\frac{(\alpha+1)\alpha^2 c}{A(u)}t^{-\alpha-2}
\end{equation*}
Moreover, for any $t>0$ the map $(c,u)\mapsto \psi_{c,u}'(t)$ is $C^{1}$.
By the previous formulae and some computations, we derive the following result:
\begin{lemma}\label{lem:H1}
For each $u\in X\setminus\{0\}$ the system $\psi_{c,u}'(t)=\psi_{c,u}''(t)=0$ has a unique solution $(t(u),c(u))$ given by
	\begin{equation}\label{RayleighCC}
		t(u)=\left(\frac{\eta-\alpha}{\beta-\alpha}\frac{N(u)}{B(u)}\right)^{\frac{1}{\beta-\eta}}, \quad	c(u)=-\frac{(\eta-\alpha)(\beta-\eta)}{\eta\beta\alpha}\left(\frac{\eta-\alpha}{\beta-\alpha}\right)^{\frac{\eta}{\beta-\eta}}\frac{N(u)^{\frac{\beta}{\beta-\eta}}}{B(u)^{\frac{\eta}{\beta-\eta}}}.
	\end{equation}
	Moreover
	\begin{enumerate}[label=(\roman*),ref=\roman*]
		\item if $c\in(c(u),0)$, then $\psi_{c,u}$ has exactly two  critical points  $0<t^+(c,u)<t^-(c,u)$, both of Morse type, the first one being a local minimizer and the second one a local maximizer. Moreover $0<\psi_{c,u}(t^+(c,u))<\psi_{c,u}(t^-(c,u))$.
		\item If $c=c(u)$, then $\psi_{c,u}$ is decreasing and has exactly one critical point  $0<t^0(c,u)$, which corresponds to an inflection point.
		\item If $c<c(u)$, then $\psi_{c,u}$ has no critical points, is decreasing and satisfies 
		$$ \lim_{t\to 0^+}\psi_{c,u}(t)=+\infty\ \text{ and }\  \lim_{t\to+ \infty}\psi_{c,u}(t)=-\infty.$$
		\item If $c\ge 0$, then  $\psi_{c,u}$ has a unique critical point $t^-(c,u)>0$ which is a global maximizer of Morse type. 
	\end{enumerate}
	Furthermore, whenever defined, $t^\pm(c,u)$ satisfy
		\begin{equation}\label{dp1}
		\frac{\eta-\alpha}{\eta}N(u)t^\pm(c,u)^{\eta}-\frac{\beta-\alpha}{\beta}B(u)t^\pm(c,u)^{\beta}+\alpha c=0
	\end{equation}
and	
\begin{equation}\label{dp111}
\frac{(\eta-\alpha)(\eta-\alpha-1)}{\eta}N(u)t^\pm(c,u)^{\eta}-\frac{(\beta-\alpha)(\beta-\alpha-1)}{\beta}B(u)t^\pm(c,u)^{\beta}-\alpha(\alpha+1) c \gtrless 0.
\end{equation}
 In particular we have
	\begin{equation*}
		t^-(0,u)=\left(\frac{\beta}{\eta} \frac{\alpha-\eta}{ \alpha-\beta}\frac{N(u)}{B(u)}\right)^{\frac{1}{\beta-\eta}}.
	\end{equation*}
\end{lemma}

Moreover we also get

\begin{lemma}\label{lem:H1C2Rayleigh}
	The functional $c$ defined in \eqref{RayleighCC} belongs to $C^1(X \setminus \{0\})$. Moreover it is $0$-homogeneous and bounded from above by a negative constant, i.e.
	$c^*:=\displaystyle \sup_{u\in X\setminus\{0\}}c(u)<0$.
\end{lemma}
\begin{proof} Indeed, it is clear that $c \in C^1(X \setminus \{0\})$, while the $0$-homogeneity follows by \eqref{1st-1}, and the boundedness from above is a consequence of \eqref{1st-3}.
\end{proof}

The next result follows from Lemma \ref{lem:H1} and shows that \eqref{H1} holds with two choices of $\mathcal{I}$:
\begin{lemma}\label{lem:H1C2}
For any $c\in (c^*,0)$ and $u\in X\setminus\{0\}$, the map $\psi_{c,u}$ satisfies  Lemma \ref{lem:H1} (1), while for any $c \ge 0$ and $u\in X\setminus\{0\}$, it satisfies  Lemma \ref{lem:H1} (4). In particular \eqref{H1} holds with $\mathcal{I}=(c^*,0)$ as well as with $\mathcal{I}=(c^*,+\infty)$.
\end{lemma}

Consequently Lemma \ref{l0} \eqref{I0-i} applies
and the maps $(c,u)\mapsto t^+(c,u),t^-(c,u)$ belong to $C^{1}((c^*,0) \times (X\setminus\{0\}))$ and $C^{1}((c^*,+\infty) \times (X\setminus\{0\}))$, respectively.
Let us obtain some further properties of these maps.

\begin{lemma}\label{l1} 
The following statements hold.
	\begin{enumerate}[label=(\roman*),ref=\roman*]

		\item\label{l1-1} 
		The map $c \mapsto t^-(c,u)$ is increasing in $(c^*,+\infty)$.
		Moreover, 
		\begin{equation}\label{bt-}
			t^-(c,u)> \left(\frac{\eta-\alpha}{\beta-\alpha}\frac{N(u)}{B(u)}\right)^{\frac{1}{\beta-\eta}},\ \forall c>c^*, u\in S.
		\end{equation}
	In particular, there exists $C'>0$ such that $t^-(c,u)\geq C'$ for every $u  \in S$ and $c >c^*$. \smallskip
	\item\label{l1-2} $t^-(c,u)\to +\infty$ as $c \to +\infty$, uniformly in $S$.
		\item\label{l1-3} 
	The map $c \mapsto t^+(c,u)$ is decreasing in $(c^*,0)$.
	Moreover, 
	\begin{equation}\label{bt+}
		t^+(c,u)<\left(-\frac{\alpha\beta\eta c}{(\beta-\eta)(\eta-\alpha)}\frac{1}{N(u)}\right)^{\frac{1}{\eta}},\ \forall c\in(c^*,0), u\in S,
	\end{equation}
In particular, there exists $C'>0$ such that $t^+(c,u)\le  C'$ for every $u  \in S$ and $c\in(c^*,0)$. \smallskip
	\item\label{l1-4} $t^+(c,u)\to 0$ as $c \to 0^-$, uniformly in $S$.
\end{enumerate}
\end{lemma} 
\begin{proof}

\eqref{l1-1}
 Note that $\Phi_\mu=I_1-\mu I_2$ with $I_2(u):=\frac{1}{\alpha}A(u)$, so that $I_2'(tu)u=t^{\alpha-1}A(u)>0$, 
 and  by   Lemma \ref{l0}\eqref{I0-i}  it follows that  the map $c \mapsto t^-(c,u)$ is increasing in $(c^*,\infty)$. To conclude the proof, note from \eqref{dp1} and \eqref{dp111} that
$(\eta-\alpha)N(u)t^-(c,u)^{\eta}-(\beta-\alpha)B(u)t^-(c,u)^{\beta} <0$,
which yields \eqref{bt-}.


\smallskip

\eqref{l1-2} By \eqref{1st-3} we know that $B$ is bounded on $S$, so the claim follows from \eqref{dp1}.

\eqref{l1-3} By Lemma \ref{l0} it follows that  the map $c \mapsto t^+(c,u)$ is decreasing in $(c^*,0)$. To conclude the proof, we infer from \eqref{dp1} and \eqref{dp111} that
$\frac{(\eta-\alpha)(\eta-\beta)}{\eta}N(u)t^+(c,u)^{\eta}-\beta\alpha c>0$,
which yields \eqref{bt+}.

\smallskip 
\eqref{l1-4} It follows from \eqref{l1-3} just proved.
\end{proof}

Let us consider now the maps $\Lambda^\pm$. First note that by Lemma \ref{lem:H1} we have $$\Lambda^+(c,u)=\psi_{c,u}(t^+(c,u))<\psi_{c,u}(t^-(c,u))=\Lambda^-(c,u), \quad \forall u \in S, c \in (c^*,0).$$ Moreover
from
 $\psi_{c,u}'(t^\pm(c,u))=0$ we have
 \begin{eqnarray}
\Lambda^\pm(c,u)&=&\frac{\beta-\eta}{\beta-\alpha}\frac{\alpha N(u)}{\eta A(u)}t^\pm(c,u)^{\eta-\alpha}-\frac{\alpha c}{A(u)}\frac{\beta}{\beta-\alpha}t^\pm(c,u)^{-\alpha} \nonumber\\
&=&\frac{\alpha t^\pm(c,u)^{-\alpha}}{(\beta-\alpha)A(u)}\left(\frac{\beta-\eta}{\eta} N(u)t^\pm(c,u)^{\eta}- \beta c\right).
\label{1st-Lambda1}
\end{eqnarray}
for the corresponding values of $c$ and $u \in X \setminus \{0\}$.
In particular,
\begin{equation*}
\Lambda^-(0,u)=C_{\alpha,\beta,\eta} \frac{N(u)^{\frac{\beta-\alpha}{\beta-\eta}}}{A(u)B(u)^{\frac{\eta-\alpha}{\beta-\eta}}} \quad \text{for some} \quad C_{\alpha,\beta,\eta}>0.
\end{equation*}

The following relations between $\Lambda^{\pm}$, $t^{\pm}$ and $N$, $A$, and $B$ shall be used repeatedly in the sequel:
	\begin{lemma}\label{ls} Let  $S_0\subset S$.
	\begin{enumerate}[label=(\roman*),ref=\roman*]
\item\label{ls:i} Let $K\subset (c^*,\infty)$ be a compact set. Then
		$$\sup_{K\times S_0}\Lambda^-<+\infty \Longleftrightarrow \sup_{ K\times S_0}t^-<+\infty \Longleftrightarrow \sup_{S_0}N<+\infty, \inf_{S_0}A>0 \mbox{ and } \inf_{S_0}B>0 .$$
\item \label{ls:ii}	Let $K\subset (c^*,0)$ be a compact set.	Then
$$\sup_{ K\times S_0}\Lambda^+<+\infty \Longleftrightarrow \sup_{S_0}N<+\infty, \inf_{S_0}A>0 \mbox{ and } \inf_{S_0}B>0.$$
Moreover, in such case we also have $\displaystyle \inf_{ K\times S_0}t^+>0$, so that $t^+$ is bounded and away from zero in $K\times S_0$. 
\end{enumerate}			
	\end{lemma}
	\begin{proof} 
	\eqref{ls:i}
 From \eqref{1st-Lambda1} it is clear that if $\Lambda^-$ is bounded from above on $K\times S_0$ then $t^-$ is bounded therein, which implies by Lemma \ref{l1} \eqref{l1-1} that that $N$ is bounded and $A,B$ are away from zero on $S_0$. Now the latter assertion yields, by \eqref{dp1}, that
\begin{equation*}
C_1t^-(c,u)^\eta-C_2t^-(c,u)^\beta+\alpha c \ge 0, \quad  \forall (c,u)\in K\times S_0,
\end{equation*}
which implies that  $t^-$ is bounded on $K\times S_0$. Since it is also away from zero, we deduce that $\Lambda^-$ is bounded from above on $K\times S_0$.

\smallskip
	\eqref{ls:ii} Now \eqref{1st-Lambda1} shows that if $\Lambda^+$ is bounded from above on $K\times S_0$ then $t^+$ is away from zero therein, $A$ and $B$ are away from zero in $S_0$, and $N$ is bounded therein.
Finally, the latter assumption and \begin{equation*}
C_1t^+(c,u)^\eta-C_2t^+(c,u)^\beta+\alpha c \ge 0, \quad  \forall (c,u)\in K\times S_0,
\end{equation*} show that $t^+$ is away from zero on $K\times S_0$, and consequently, by \eqref{1st-Lambda1}, $\Lambda^+$ is bounded from above therein.
	\end{proof}

\begin{lemma} \label{pl} 
The following statements hold.
\begin{enumerate}[label=(\roman*),ref=\roman*]
\item\label{pl-i} For every $c> c^*$ the map $u \mapsto \Lambda^-(c,u)$ is bounded from below on $S$. For every $c^*<c<0$ the map $u \mapsto \Lambda^+(c,u)$ is bounded from below on $S$, by a positive constant.
\item\label{pl-ii}  For every $u\in X \setminus\{0\}$ the maps $c\mapsto \Lambda^+(c,u),\Lambda^-(c,u)$ are decreasing in $(c^*,0)$ and $(c^*,+\infty)$, respectively.
\item\label{pl-iii} If $\Lambda^-(0,u)$ is bounded in  $S_0\subset S$
then $\displaystyle \lim_{c \to +\infty}\Lambda^-(c,u)= -\infty$ uniformly in $S_0$.
\item\label{pl-iv} If  $\Lambda^+(\overline{c},u)$ is bounded in $S_0\subset S$ for some $\overline{c}\in(c^*,0)$ then $\displaystyle \lim_{c \to 0^-}\Lambda^+(c,u)= 0$ uniformly in $S_0$.
\end{enumerate}
\end{lemma}

\begin{proof}
\eqref{pl-i} Assume by contradiction that there exists $\{u_n\} \subset S$ such that $\Lambda^-(c,u_n) \to -\infty$. Then Lemma \ref{ls} implies that $\{t^-(c,u_n)\}$ is bounded. By Lemma \ref{l1} \eqref{l1-1} we also know that  $\{t^-(c,u_{n})\}$ is  away from zero. Thus, by \eqref{1st-Lambda1}, it follows that $A(u_{n})\to 0$, which contradicts Lemma \ref{ls}. The first assertion is then proved. Now,
from \eqref{1st-Lambda1} we deduce that
$\Lambda^+(c,u) \geq -C_1c t^+(c,u)^{-\alpha}$ for some $C_1>0$, and since $t^+$ is bounded, we infer the second assertion.
\smallskip

\eqref{pl-ii}  It follows from  Lemma \ref{l0} \eqref{I0-ii} and the fact that $I_2>0$.

\smallskip

\eqref{pl-iii}
By Lemma \ref{ls} we know that if $\Lambda^-(0,u)$ is bounded in $S_0$, then $N$ is bounded  and $A,B$ are  away from zero in $S_0$. Therefore, 
by \eqref{1st-psi},
 \begin{equation*}
 	\psi_{c,u}(t)\le\varphi_c(t):= C_1t^{\eta-\alpha}- C_2t^{\beta-\alpha}-C_3ct^{-\alpha},\quad \forall t>0, 
	\ \ \forall u \in S_0,
 \end{equation*}
where $C_1,C_2,C_3$ are positive constants. Note that $\displaystyle \max_{t>0}\varphi_c(t)\to -\infty $ as $c\to +\infty$,  and since
\begin{equation*}
	\Lambda^-(c,u)=\psi_{c,u}(t^-(c,u))\le \varphi_c(t^-(c,u))\le  \max_{t>0}\varphi_c(t),
\end{equation*}
we obtain the desired conclusion. 

\smallskip

\eqref{pl-iv} Suppose that $\Lambda^+(\overline{c},u)$ is bounded in $S_0\subset S$.
By Lemma \ref{ls} we know that $A$ is away from zero and $N$ is bounded in $S_0$. 
 Therefore, by \eqref{dp1} there exist $C_1,C_2>0$ such that
\begin{equation*}
C_1t^+(c,u)^{\eta}-C_2t^+(c,u)^{\beta}+\alpha c\ge 0, \quad \forall  c\in(\overline{c},0), u\in S_0,
\end{equation*}
 and thus
\begin{equation*}
	C_1t^+(c,u)^{\eta-\alpha}-C_2t^+(c,u)^{\beta-\alpha}\ge -\alpha \frac{c}{t^+(c,u)^{\alpha}}>0,\quad \forall c\in(\overline{c},0), u\in S_0.
\end{equation*}
By  Lemma \ref{l1} \eqref{l1-3}, it follows that $\frac{c}{t^+(c,u)^{\alpha}}\to 0$, as $c\to 0^-$, uniformly in $S_0$. As a consequence, by \eqref{1st-Lambda1}, we conclude that 
 $\displaystyle \lim_{c \to 0^-}\Lambda^+(c,u)= 0$ uniformly in $S_0$.
\end{proof}


Next we prove that the maps $u \mapsto \widetilde{\Lambda^{\pm}}(c,u)$ satisfy the Palais-Smale condition. In the sequel $I_+:=(c^*,0)$ and $I_-=(c^*,+\infty)$.

\begin{proposition}\label{PS} 
Suppose that $c_n\to c \in \mathcal{I} _\pm$
and $\{u_{n}\}\subset S$ are such that $\frac{\partial \widetilde{\Lambda^{\pm}}}{\partial u}(c_n,u_n)\to 0$ and $\{\widetilde{\Lambda^{\pm}}(c_n,u_n)\}$ is bounded. Then 
$\{u_n\}$ has  a convergent subsequence. In particular, the map $u \mapsto \widetilde{\Lambda^{\pm}}(c,u)$ satisfies the Palais-Smale condition. 
\end{proposition}
\begin{proof} Write, for simplicity, $\Lambda:=\Lambda^{\pm}$ and $t(c,u):=t^\pm(c,u)$.  Let $t_n:=t(c_n,u_n)$. We can assume that  $u_n\rightharpoonup u$ in $X$.
By Lemma \ref{l1} and Lemma \ref{ls} we deduce that $\{t_n\}$ is bounded and away from zero, and $u\neq 0$. 	Hence, up to a subsequence, the assumptions of Lemma \ref{psg} are satisfied (see also Remark \ref{rps}), and we obtain the desired conclusion.
\end{proof}

By Lemma \ref{lem:H1C2}, Lemma \ref{pl} and Proposition \ref{PS} we are in position to apply Theorem \ref{THM1}
to show that for any $c$ the sequences
$$\mu_{n,c}^- = \inf_{F\in \mathcal F_{n}} \sup_{u\in F}\widetilde{\Lambda^-}(c,u),\quad \mbox{if} \quad c>c^*,$$ and 
$$ \mu_{n,c}^+ = \inf_{F\in \mathcal F_{n}} \sup_{u\in F}\widetilde{\Lambda^+}(c,u), \quad \mbox{if} \quad  c\in(c^*,0)$$
give rise to sequences $\{u_{n,c}\}$, $\{v_{n,c}\}$ of critical points of the families $\Phi_{\mu_{n,c}^{-}}$ and $\Phi_{\mu_{n,c}^{+}}$, respectively, and having energy equal to $c$. Note that by Lemma \ref{pl} we have $\mu_{n,c}^+>0$ for every $n \in \mathbb{N}$ and $c\in(c^*,0)$.

\begin{corollary}\label{cii} The following holds.
\begin{enumerate}	[label=(\roman*),ref=\roman*]
\item\label{cii:i} Let $c>c^*$. Then $\mu_{n,c}^- \to+ \infty$ and $\|u_{n,c}\| \to +\infty$ as $n\to+ \infty$.
\item \label{cii:ii} Let $c \in (c^*,0)$. Then $\mu_{n,c}^+ \to+ \infty$ and either $v_{n,c} \to 0$ or $v_{n,c} \rightharpoonup 0$ and $v_{n,c} \not \to 0$ as $n\to+ \infty$. In particular, the latter case holds if $N$ is bounded on bounded sets.
\end{enumerate}	
\end{corollary}	

\begin{proof}
\eqref{cii:i} 
 By Lemma \ref{ls} we know that $\Lambda^{-}(c,w_n) \to \infty$ if $w_n \rightharpoonup 0$ in $X$, so Theorem \ref{THM1} yields the first assertion. Moreover, writing $u_{n,c}=t^-(c,w_n)w_n$ with $w_n \in S$, we have $\|u_{n,c}\|=t^-(c,w_n)$ and $\Lambda^-(c,w_n)=\mu_{n,c}^- \to \infty$. By Lemma \ref{ls} we deduce that $t^-(c,w_n)\to \infty$. 
 
 \smallskip
 
\eqref{cii:ii}  Lemma \ref{ls} yields the first assertion. 
Writing $v_{n,c}=t^+(c,w_n)w_n$, by Lemma \ref{ls} we have either $N(w_n) \to \infty$ or $A(w_n) \to 0$. In the first case we have $t^+(c,w_n)\to 0$ by \eqref{bt+}, so that $v_{n,c} \to 0$. In the second case $w_n \rightharpoonup 0$ in $X$. Since $t^+(c,w_n)$ is bounded by Lemma \ref{l1}, we see that $v_{n,c} \rightharpoonup 0$ in $X$. Finally, \eqref{dp1} shows that $t^+(c,w_n) \not \to 0$, so $v_{n,c} \not \to 0$ in $X$.
\end{proof}

\begin{proposition} \label{pdm} There holds $\mu_{n,c}^+<\mu_{n,c}^-$ for every $n\in \mathbb{N}$ and $c\in (c^*,0)$. 
\end{proposition}

The proof of this result relies on the following lemma:

\begin{lemma}\label{secondd} Suppose that there exists a sequence $(s_n,u_n)\in \mathbb{R}\times S$ such that $t^+(c,u_n)\le s_n<t^+(c,u_n)+1/n$, $\Lambda^+(c,u_n)$ is bounded and $\psi_{c,u_n}''(s_n)=o_n(1)$. Then $c\le c^*$.
\end{lemma}
\begin{proof} Let us write $t_n=t^+(c,u_n)$. By Lemma \ref{l1} and Lemma \ref{ls} we know that $\{t_n\}$, $\{N(u_n)\}$ and $\{A(u_n)\}$ are bounded and away from zero, so we can assume that $t_n,s_n \to t>0$ and $u_n \rightharpoonup u\neq 0$. Thus
\small{	\begin{eqnarray*}
			\frac{\eta-\alpha}{\eta}N(u_n)s_n^{\eta}-\frac{\beta-\alpha}{\beta}B(u_n)s_n^{\beta}+\alpha c 
		&=&\frac{\eta-\alpha}{\eta}N(u_n)(t_n+o_n(1))^{\eta}-\frac{\beta-\alpha}{\beta}B(u_n)(t_n+o_n(1))^{\beta}+\alpha c \\
		&=& t_n^{\alpha+1}A(u_n)\psi_{c,u_n}'(t_n)+o_n(1)=o_n(1),\quad \forall n,
	\end{eqnarray*}}
	which combined with $\psi_{c,u_n}''(s_n)=o_n(1)$ yields
	\begin{equation*} \label{pu}
	\left\{
	\begin{array}
	[c]{lll}%
	\displaystyle\frac{\eta-\alpha}{\eta}N(u_n)s_n^{\eta}-\frac{\beta-\alpha}{\beta}B(u_n)s_n^{\beta}+\alpha c =o_n(1),\\
	\displaystyle\frac{(\eta-\alpha)(\eta-\alpha-1)}{\eta}N(u_n)s_n^{\eta}-\frac{(\beta-\alpha)(\beta-\alpha-1)}{\beta}B(u_n)s_n^{\beta}-\alpha(\alpha+1) =o_n(1).
	\end{array}
	\right. 
	\end{equation*}
	Solving this system in the variables $(s_n,c)$ we conclude that
	$s_n=\left(\frac{\eta-\alpha}{\beta-\alpha}\frac{N(u_n)}{B(u_n)}+o_n(1)\right)^{\frac{1}{\beta-\eta}}$,	and (recall \eqref{RayleighCC}) $c=\lim_{n\to+\infty} c(u_n)\leq c^*$.
\end{proof}
\begin{corollary}\label{secondd1} Fix $c\in (c^*,0)$ and suppose that $\Lambda^+(c,\cdot)$ is bounded over $S_0\subset S$. Then there exists $D,\delta>0$ such that $\psi_{c,u}''(s)\ge D$ for all $(s,u)\in[t^+(c,u),t^+(c,u)+\delta]\times S_0$.
\end{corollary}
\begin{proof} On the contrary we can find a sequence $(s_n,u_n)\in \mathbb{R}\times S_0$ such that $t^+(c,u_n)\le s_n<t^+(c,u_n)+1/n$, $\Lambda^+(c,u_n)$ is bounded and $\psi_{c,u_n}''(s_n)=o_n(1)$, which contradicts Lemma \ref{secondd}.
\end{proof}

\begin{proof}[Proof of Proposition \ref{pdm}] Indeed, since $\Lambda^+(c,u)<\Lambda^-(c,u)$ we have $\mu_{n,c}^+\le \mu_{n,c}^-$. This inequality also implies that
	$\displaystyle \mu_{n,c}^+=\inf_{F\in\widetilde{\mathcal{F}}_n}\sup_{u\in F}\Lambda^+(c,u)
	$ and $\displaystyle \mu_{n,c}^-=\inf_{F\in \widetilde{\mathcal{F}}_n}\sup_{u\in F}\Lambda^-(c,u)$,
	where	$\widetilde{\mathcal{F}}_n=\{F\in \mathcal{F}_n:\ \displaystyle\sup_{u\in F}\Lambda^-(c,u)\le \mu_{n,c}^-+1 \}$.	Let $S_0=\{u\in S: \Lambda^-(c,u)\le \mu_{n,c}^-+1\}$. We write $t_0=t^+(c,u)$, $t_1=t^-(c,u)$ and $t_\delta=t_0+\delta$. From Corollary \ref{secondd1} and the mean value theorem we infer that
	$\psi_{c,u}(t_\delta)\ge	\psi_{c,u}(t_0) +D\delta^2$, so that
	$\psi_{c,u}(t_1)-\psi_{c,u}(t_0)> D\delta^2$, for any $u\in S_0$.
	Hence
	$\displaystyle 	\sup_{u\in F}\Lambda^+(c,u)\leq \sup_{u\in F}\Lambda^-(c,u)-D\delta^2$, for any $F\in \widetilde{\mathcal{F}}_n$,
	and then
	$\mu_{n,c}^+\leq	\mu_{n,c}^--D\delta^2$,
	which completes the proof.
\end{proof}

Let us proceed with the study of the behavior of $\mu_{n,c}^\pm$ with respect to $c$.

\begin{lemma}\label{lip} 
	For every $n \in \mathbb{N}$ the following properties hold:
	\strut
	\begin{enumerate}[label=(\roman*),ref=\roman*]
		\item\label{lip-0}  The maps $c \mapsto \mu_{n,c}^\pm$ are decreasing and locally Lipschitz continuous in $I_\pm$. 
		\item\label{lip-i} $\displaystyle \lim_{c\to +\infty}\mu_{n,c}^-=-\infty$ and $\displaystyle \lim_{c\to 0^-}\mu_{n,c}^+=0$.
	\end{enumerate}
\end{lemma}

\begin{proof} 
	\eqref{lip-0} Indeed, by \eqref{dl} and Lemma \ref{ls} we see that \eqref{H3} is satisfied by $\Lambda^{\pm}$, whereas Proposition \ref{PS} shows that \eqref{PSG} is satisfies by these functionals, so Theorem \ref{THM2} yields the conclusion.
	
	\smallskip
	
	\eqref{lip-i} It suffices to pick some $F \in \mathcal{F}_n$ such that $N$ is bounded and $A,B$ are away from zero in $F$. By Lemma \ref{ls} we deduce that $\Lambda^-(0,u)$ and $\Lambda^+(\overline{c},u)$ are bounded in $F$, for any $\overline{c} \in (c^*,0)$. Thus Lemma \ref{pl} implies that $\displaystyle \lim_{c \to +\infty}\Lambda^-(c,u)= -\infty$ and $\displaystyle \lim_{c \to 0^-}\Lambda^+(c,u)= 0$ uniformly in $F$, which yields the desired conclusions.
\end{proof}

Next we show that bifurcation occurs at $(\mu,+\infty)$ for every $\mu \in \mathbb{R}$, and at $(\mu,0)$ for every $\mu >0$: 

\begin{corollary}\label{ccbif} The following statements hold.
	\begin{enumerate} [label=(\roman*),ref=\roman*]
		\item\label{ccbif:i} For every $\mu\in\mathbb{R}$ there exist two increasing sequences $\{c_{k}\} \subset (0,+\infty)$ and $\{n_k\} \subset \mathbb{N}$ such that $\mu_{n_k,c_{k}}^-=\mu$ for every $k\in \mathbb{N}$, and $c_{k},n_k,\|v_{n_k,c_k}\| \to+ \infty$ as $k \to +\infty$.
		
		\item \label{ccbif:ii} For every $\mu>0$ there  exist two increasing sequences $\{c_{k}\} \subset (c^*,0)$ and $\{n_k\} \subset \mathbb{N}$ such that $\mu_{n_k,c_{k}}^+=\mu$ for every $k\in \mathbb{N}$,  $n_k\to+ \infty$,  $c_{k}\to 0^-$, and $u_{n_k,c_k}\to 0$  in  $X$ as $k \to+ \infty$.
		
	\end{enumerate}
\end{corollary}

\begin{proof}
	We can argue as in the proof of Corollary \ref{cbif} to obtain the sequences $\{c_{k}\}$ and $\{n_k\}$. Moreover,  writing $v_k=v_{n_k,c_k}$ we have $v_k=t_kw_k$ where $t_k=t^-(c_{n_k},w_k)$ and $w_k \in S$ for every $k$. Since $c_k\to +\infty$, it follows from Lemma \ref{l1} \eqref{l1-2} that $\|v_k\|=t_k\to+ \infty$, which proves \eqref{ccbif:i}.
	A similar argument applied to $u_k=u_{n_k,c_k}$ yields $u_k=t_kw_k$ with $t_k=t^+(c_{n_k},w_k)$, and Lemma \ref{l1} \eqref{l1-4} shows that $u_k \to 0$ in $X$, giving \eqref{ccbif:ii}.
\end{proof}

\begin{theorem} \label{ts2} Under the previous conditions there exists $c^*<0$ such that: 
	\begin{enumerate}[label=(\roman*),ref=\roman*]
		\item\label{ts2:i} 
		 For any $c>c^*$ there exists a sequence $\{(\mu_{n,c}^-,u_{n,c})\}  \subset \mathbb{R} \times X\setminus\{0\}$ such that $(\mu_{n,c}^-,\pm u_{n,c})$ solves \eqref{ef} for every $n \in \mathbb{N}$.  Moreover:
		\begin{enumerate}
			\item  $\{\mu_{n,c}^-\}$ is non-decreasing, and $\mu_{n,c}^-,\|u_{n,c}\| \to +\infty$ as $n\to +\infty$.
			\item If $c>0$ and $\mu<\mu_{1,c}^-$ then \eqref{ef} has no solution.
		\end{enumerate}
		
		\item\label{ts2:ii} For any $c\in (c^*,0)$  there exists a sequence $\{(\mu_{n,c}^+,v_{n,c})\}  \subset (0,\infty) \times X\setminus\{0\}$ such that $(\mu_{n,c}^+,\pm v_{n,c})$ solves \eqref{ef} for every $n \in \mathbb{N}$. Moreover:
		\begin{enumerate}
			\item $\{\mu_{n,c}^+\}$ is non-decreasing, and $\mu_{n,c}^+ \to +\infty$ and $v_{n,c} \rightharpoonup 0$ as $n\to +\infty$.
			\item $\mu_{n,c}^+<\mu_{n,c}^-$ for every $n$.
			\item If $\mu<\mu_{1,c}^+$ then \eqref{ef} has no solution.
		\end{enumerate}

		\item\label{ts2:iii} For every $n\in \mathbb{N}$ the map $c \mapsto \mu_{n,c}^-$ is locally Lipschitz continuous and decreasing in $(c^*,\infty)$, and $\displaystyle \lim_{c\to +\infty}\mu_{n,c}^-=-\infty$. 
		
		\item \label{ts2:iv}For every $n\in \mathbb{N}$ the map $c \mapsto \mu_{n,c}^+$ is locally Lipschitz continuous and decreasing in $(c^*,0)$, and $\displaystyle \lim_{c\to 0^-}\mu_{n,c}^+=0$.

		\item\label{ts2:v} For every $\mu\in \mathbb{R}$ there are infinitely many $u_n\in X$ such that 
		$\Phi_\mu'(\pm u_n)=0$, and $\Phi_\mu(u_n)\to \infty$ and $\|u_n\|\to +\infty$ as $n \to \infty$, so $(\mu,+\infty)$ is a bifurcation point for any  $\mu\in \mathbb{R}$.
		
		\item \label{ts2:vi}For every $\mu>0$ there are infinitely many  $v_n\in X$ such that 
		$\Phi_\mu'(\pm v_n)=0$ and $\Phi_\mu(v_n)\to 0^-$ and $v_n\to 0$ in $X$  as $n \to +\infty$, so $(\mu,0)$ is a bifurcation point for any  $\mu>0$.
	\end{enumerate}	
\end{theorem}

\begin{proof}
\eqref{ts2:i} and \eqref{ts2:ii} follow from Corollary \ref{cii}, Lemma \ref{pl}, and Propositions \ref{PS} and \ref{pdm}, combined with Theorems \ref{THM0} and \ref{THM1}, whereas \eqref{ts2:iii} and \eqref{ts2:iv} follow from Lemma \ref{lip}. Finally, \eqref{ts2:v} and \eqref{ts2:vi} are consequences of Corollary \ref{ccbif}.
\end{proof}

\subsection{Proof of Theorems \ref{THMAP2} and \ref{EC2}}
It is clear that $N(u)=\int_\Omega |\nabla u|^p$, $A(u)=\int_\Omega |u|^q$, and $B(u)=\int_\Omega |u|^r$ satisfy the conditions
\eqref{1st-1}-\eqref{1st-5} of this section with $\alpha=q$, $\eta=p$ and $\beta=r$,
so Theorem \ref{ts2} yield all the conclusions, except Theorem \ref{THMAP2} (1)-(b) and Theorem \ref{THMAP2} (2)-(c), which follow from the fact that \eqref{quasi} has no positive solution for $\mu>\mu_1$, the first eigenvalue of the Dirichlet $p$-Laplacian,  cf. \cite[Theorem 2.1]{AH}.

\medskip


\section{A third class of functionals}
\label{sec:6}
\medskip

The last functional we study  is
\begin{equation*}
	\Phi_\mu(u):=\frac{1}{\eta}N(u) +\frac{\mu}{\alpha}A(u)-\frac{1}{\beta}B(u)
\end{equation*}
where 
$1<\eta<\beta<\alpha$, and $N,A,B \in C^1(X)$ are even functionals satisfying:
\begin{enumerate}
	\item\label{2nd-1} $N,A,B$ are  $\eta$-homogeneous, $\alpha$-homogeneous and $\beta$-homogeneous, respectively.
	\item\label{2nd-2} $A(u),B(u)>0$ for any $u \neq 0$.
	\item\label{2nd-3} There exists $C>0$ such that $ N(u)\geq C^{-1}\|u\|^{\eta}$ and $B(u)\leq C\|u\|^{\beta}$ for all $u \in X$. 
	\item\label{2nd-4} $A'$ and $B'$ are completely continuous, i.e. $A'(u_n) \to A'(u)$ and $B'(u_n) \to B'(u)$ in $X^*$ if $u_n \rightharpoonup u$ in $X$.

	\item\label{2nd-6} $N$ is weakly lower semicontinuous and there exists $C>0$ such that 
		$$(N'(u)-N'(v))(u-v)\geq C(\|u\|^{\eta-1}-\|v\|^{\eta-1})(\|u\|-\|v\|) \quad \forall u,v \in X.$$ 
\end{enumerate}

We will work in this section by assuming implicitly the above conditions.

Furthermore, we assume the following condition on the family $\Phi_\mu$, which is satisfied, for instance, if $\Phi_\mu$ is coercive for every $\mu>0$:\\
\begin{enumerate}[label=(H\arabic*),ref=H\arabic*,start=4]
\item\label{H4} For any $c,a>0$ there exists $M>0$ such that $\displaystyle \cup_{\mu\ge a} \{u\in X: \Phi_{\mu}(u)\leq c\} \subset B(0,M)$. Hereafter $B(0,M)$ is the ball entered in $0$ of radius $M$ in $X$.
\end{enumerate}
\medskip
Let then $c \in \mathbb{R}$ be fixed. Note that for $u\in X\setminus\{0\}$, we have 
\begin{equation*}
	\mu(c,u):=\displaystyle \frac{\alpha}{A(u)}\left(-\frac{1}{\eta}N(u)+\frac{1}{\beta}B(u)+c\right).
\end{equation*}
Then
\begin{equation*}
	\psi_{c,u}(t)=-\frac{\alpha N(u)}{\eta A(u)}t^{\eta-\alpha}+\frac{\alpha B(u)}{\beta A(u)}t^{\beta-\alpha}+\frac{\alpha c}{A(u)}t^{-\alpha},\quad t>0
\end{equation*}
is a $C^{2}$ map, with
\begin{equation*}
	\psi_{c,u}'(t)=(\alpha-\eta)\frac{\alpha N(u)}{\eta A(u)}t^{\eta-\alpha-1}- (\alpha-\beta)\frac{\alpha B(u)}{\beta A(u)}t^{\beta-\alpha-1}-\frac{\alpha^2 c}{A(u)}t^{-\alpha-1}.
\end{equation*}
Moreover, for any $t>0$ the map $(c,u)\mapsto \psi_{c,u}'(t)$ is $C^{1}$.

A simple analysis of $\psi_{c,u}$ provides us  the following result:
\begin{lemma}\label{techn1}
	For each $u\in X\setminus\{0\}$ the system $\psi_{c,u}'(t)=\psi_{c,u}''(t)=0$ has a unique solution $(t(u),c(u))$ given by
	\begin{equation}\label{RayleighSP}
		t(u)=\left(\frac{\alpha-\eta}{\alpha-\beta}\frac{N(u)}{B(u)}\right)^{\frac{1}{\beta-\eta}}, \quad	c(u)=\frac{(\alpha-\eta)(\beta-\eta)}{\eta\beta\alpha}\left(\frac{\alpha-\eta}{\alpha-\beta}\right)^{\frac{\eta}{\beta-\eta}}\frac{N(u)^{\frac{\beta}{\beta-\eta}}}{B(u)^{\frac{\eta}{\beta-\eta}}}.
	\end{equation}
	Moreover:
	\begin{enumerate}[label=(\roman*),ref=\roman*]
		\item \label{c-i}If $c\in(0,c(u))$ then $\psi_{c,u}$ has exactly two  critical points  $0<t^+(c,u)<t^-(c,u)$, both of Morse type, the first one being a local minimizer and the second one a local maximizer. Moreover $\psi_{c,u}(t^-(c,u))>0$.
		\item \label{c-ii} If $c=c(u)$ then $\psi_{c,u}$ is decreasing and has exactly one critical point  $t^0(c,u)>0$, which corresponds to an inflection point.
		\item \label{c-iii} If $c>c(u)$, then $\psi_{c,u}$ is decreasing, has no critical points, and satisfies $\displaystyle \lim_{t\to 0^+}\psi_{c,u}(t)=+\infty$ and $\displaystyle \lim_{t\to +\infty}\psi_{c,u}(t)=0$.
		\item \label{c-iv} If $c\le 0$, then  $\psi_{c,u}$ has a unique critical point $t^-(c,u)>0$, which is a global maximizer of Morse type, and $\psi_{c,u}(t^-(c,u))>0$.
	\end{enumerate}
	Furthermore, whenever defined, $t^\pm(c,u)$ satisfy
	\begin{equation}\label{dp1SP}
		\frac{\alpha-\eta}{\eta}N(u)t^\pm(c,u)^{\eta}+\frac{\beta-\alpha}{\beta}B(u)t^\pm(c,u)^{\beta}-\alpha c=0,
	\end{equation}
	and	
	\begin{equation}\label{dp111SP}
		\frac{(\alpha-\eta)(\eta-\alpha-1)}{\eta}N(u)t^\pm(c,u)^{\eta}+\frac{(\beta-\alpha)(\beta-\alpha-1)}{\beta}B(u)t^\pm(c,u)^{\beta}+\alpha(\alpha+1) c \gtrless 0.
	\end{equation}
	In particular, we have
	\begin{equation*}
		t^-(0,u)=\left(\frac{\beta}{\eta}\frac{\alpha-\eta}{\alpha-\beta}\frac{N(u)}{B(u)}\right)^{\frac{1}{\beta-\eta}}=\left(\frac{\alpha}{\eta}\frac{\beta-\eta}{\alpha-\beta}\frac{N(u)}{\Lambda^-(0,u)A(u)}\right)^{\frac{1}{\alpha-\eta}}.
	\end{equation*}
\end{lemma}
\medskip
Some properties of the functional $c$ given in \eqref{RayleighSP} are given in the next

\begin{lemma}\label{lem:H1SPRayleigh}
	The functional $c$ defined in \eqref{RayleighSP} belongs to $C^1(X \setminus \{0\})$. Moreover it is $0$-homogeneous and bounded from below by a positive constant. 
\end{lemma}
\begin{proof} Indeed, it is clear that $c \in C^1(X \setminus \{0\})$, while the $0$-homogeneity follows by \eqref{2nd-1} and the boundedness from below is a consequence of \eqref{2nd-3}.
\end{proof}

By Lemma \ref{lem:H1SPRayleigh}, it follows that
$c^*:=\displaystyle \inf_{u\in X\setminus\{0\}}c(u)>0$.

\begin{lemma}\label{lem:H1SP}
	For any $c\in (0,c^*)$ and $u\in X\setminus\{0\}$, the map $\psi_{c,u}$ satisfies  Lemma \ref{techn1} \eqref{c-i}, while for any $c \le 0$ and $u\in X\setminus\{0\}$, it satisfies  Lemma \ref{techn1} \eqref{c-iv}. In particular \eqref{H1} holds with $I=(0,c^*)$ as well as with $I=(-\infty,c^*)$.
\end{lemma}

Consequently Lemma \ref{l0} \eqref{I0-i} applies
and the maps $(c,u)\mapsto t^+(c,u),t^-(c,u)$ belong to $C^{1}((0,c^*) \times (X\setminus\{0\}))$ and $C^{1}((-\infty,c^*) \times (X\setminus\{0\}))$, respectively.
Let us obtain some further properties of these maps.

\begin{lemma}\label{l1SP} 
	The following properties hold:
\begin{enumerate}[label=(\roman*),ref=\roman*]
	\item\label{l1-1SP}	The map $c \mapsto t^-(c,u)$ is decreasing in $(-\infty,c^*)$.
	Moreover, 
	\begin{equation*}
		t^-(c,u)> \left(\frac{\alpha-\eta}{\alpha-\beta}\frac{N(u)}{B(u)}\right)^{\frac{1}{\beta-\eta}},\ \forall c<c^*, u\in S.
	\end{equation*}
	In particular, there exists $C'>0$ such that $t^-(c,u)\geq C'$ for every $u  \in S$ and $c <c^*$.
	\smallskip
	\item\label{l1-2SP} $t^-(c,u)\to +\infty$ as $c \to -\infty$, uniformly in $S$.
	\item\label{l1-3SP} 
	The map $c \mapsto t^+(c,u)$ is increasing in $(0,c^*)$.
	Moreover, 
	\begin{equation*}
		t^+(c,u)<\left(\frac{\alpha\beta\eta c}{(\beta-\eta)(\alpha-\eta)}\frac{1}{N(u)}\right)^{\frac{1}{\eta}},\ \forall c\in(0,c^*), u\in S.
	\end{equation*}
	In particular, there exists $C'>0$ such that $t^+(c,u)\le  C'$ for every $u  \in S$ and $c\in(0,c^*)$. \smallskip
	\item\label{l1-4SP} $t^+(c,u)\to 0$ as $c \to 0^+$, uniformly in $S$.
\end{enumerate}
\end{lemma} 
\begin{proof}

\eqref{l1-1}
Note that $\Phi_\mu=I_1-\mu I_2$ with $I_2(u):=-\frac{1}{\alpha}A(u)$, so that $I_2'(tu)u=-t^{\alpha-1}A(u)<0$, 
and  by   Lemma \ref{l0}\eqref{I0-i}  it follows that  the map $c \mapsto t^-(c,u)$ is decreasing in $(-\infty,c^*)$. In addition, note from \eqref{dp1SP} and \eqref{dp111SP} that
$(\alpha-\eta)N(u)t^-(c,u)^{\eta}+(\beta-\alpha)B(u)t^-(c,u)^{\beta} <0,
$
so that
$t^-(c,u)> \left(\frac{\alpha-\eta}{\alpha-\beta}\frac{N(u)}{B(u)}\right)^{\frac{1}{\beta-\eta}}
$
for any $c>c^*$ and $u \in S$. 


\smallskip

\eqref{l1-2} By \eqref{2nd-3} we know that $B$ is bounded on $S$, so the claim follows from \eqref{dp1SP}.

\smallskip

\eqref{l1-3}  Lemma \ref{l0} shows that the map $c \mapsto t^+(c,u)$ is decreasing in $(0,c^*)$. In addition, from \eqref{dp1SP} and \eqref{dp111SP} we infer that
$\frac{(\alpha-\eta)(\eta-\beta)}{\eta}N(u)t^+(c,u)^{\eta}+\beta\alpha c>0$,
i.e.
$t^+(c,u)<\left(\frac{\alpha\beta\eta c}{(\beta-\eta)(\alpha-\eta)}\frac{1}{N(u)}\right)^{{1}/{\eta}}$,
for $c\in(0,c^*)$, and $u\in S$.

\smallskip

\eqref{l1-4} It follows from \eqref{l1-3} and the fact that $N$ is bounded away from zero on $S$.
\end{proof}
Let us consider now the maps $\Lambda^\pm$. First note that from
$\psi_{c,u}'(t^\pm(c,u))=0$ we have
\begin{eqnarray}
	\Lambda^\pm(c,u)=\psi_{c,u}(t^\pm(c,u))&=&\frac{\beta-\eta}{\alpha-\beta}\frac{\alpha N(u)}{\eta A(u)}t^\pm(c,u)^{\eta-\alpha}-\frac{\alpha c}{A(u)}\frac{\beta}{\alpha-\beta}t^\pm(c,u)^{-\alpha} \nonumber\\
	&=&\frac{\alpha t^\pm(c,u)^{-\alpha}}{(\alpha-\beta)A(u)}\left(\frac{\beta-\eta}{\eta} N(u)t^\pm(c,u)^{\eta}- \beta c\right)\nonumber \\
	&=& \frac{\alpha ct^\pm(c,u)^{-\alpha}}{(\alpha-\eta)A(u)}\left(\frac{\beta-\eta}{\beta} B(u)\frac{t^\pm(c,u)^{\beta}}{c}- \eta \right)	\label{1st-Lambda1SP}
\end{eqnarray}
for the corresponding values of $c$ and $u \in X \setminus \{0\}$.
In particular,
\begin{equation*}
	\Lambda^-(0,u)=D_{\alpha,\beta,\eta} \frac{B(u)^{\frac{\alpha-\eta}{\beta-\eta}}}{A(u)N(u)^{\frac{\alpha-\beta}{\beta-\eta}}}\quad \text{ for some }D_{\alpha,\beta,\eta}>0.
\end{equation*}
	\begin{lemma}\label{lSPc} Let $S_0\subset S$. 
		 \begin{enumerate}[label=(\roman*),ref=\roman*]
		\item\label{lSPc1}  Let $K\subset (-\infty,c^*)$ be a compact set. Then
			\begin{eqnarray*}
				\inf_{(c,u)\in K\times S_0}\Lambda^-(c,u)>0 &\Longleftrightarrow& \sup_{u\in K\times S_0}t^-(c,u)<+\infty \\
				&\Longleftrightarrow& \sup_{u\in S_0}N(u)<+\infty, \inf_{u\in S_0}A(u)>0\ and \inf_{u\in S_0}B(u)>0.
			\end{eqnarray*} 
		\item\label{lSPc2} Let $K\subset (0,c^*)$ be a compact. Then
			\begin{equation*}
				\inf_{(c,u)\in K\times S_0}\Lambda^+(c,u)>-\infty \Longleftrightarrow \sup_{u\in S_0}N(u)<+\infty, \inf_{u\in S_0}A(u)>0\ and \inf_{u\in S_0}B(u)>0.
			\end{equation*}
	 Moreover any of the latter statements implies that $\inf_{u\in K\times S_0}t^+(c,u)>0$.
	\end{enumerate} 
\end{lemma}
\begin{proof} \eqref{lSPc1} Indeed, suppose that $\inf_{(c,u)\in K\times S_0}\Lambda^-(c,u)>0$, then by \eqref{H4} we have that $\sup_{(c,u)\in K\times S_0}t^-(c,u)<+\infty$ and hence by Lemma \ref{l1SP} \eqref{l1-1SP} and \eqref{2nd-4}  we conclude that  $\sup_{u\in S_0}N(u)<+\infty$, $\inf_{u\in S_0}A(u)>0$  and $\inf_{u\in S_0}B(u)>0$. 		
		
		Now assume that $\sup_{u\in S_0}N(u)<+\infty$ and $\inf_{u\in S_0}B(u)>0$.  The latter assertion yields, by \eqref{dp1SP}, that
		\begin{equation*}
			C_1t^-(c,u)^\eta-C_2t^-(c,u)^\beta-\alpha c \ge 0, \quad  \forall (c,u)\in K\times S_0,
		\end{equation*}
		which implies that $\sup_{(c,u)\in K\times S_0}t^-(c,u)<+\infty$.
		
		To conclude suppose that $\sup_{(c,u)\in K\times S_0}t^-(c,u)<+\infty$. By Lemma \ref{l1SP} \eqref{l1-1SP} and \eqref{2nd-4}  we conclude that  $\sup_{u\in S_0}N(u)<+\infty$, $\inf_{u\in S_0}A(u)>0$  and $\inf_{u\in S_0}B(u)>0$. By \eqref{1st-Lambda1SP} it is clear that $\inf_{(c,u)\in (K\cap (-\infty,0])\times S_0}\Lambda^-(c,u)>0$. Moreover since by Lemma \ref{plSP} \eqref{pl-iiSP} we have that $\Lambda^-(c,u)>\Lambda^-(0,u)$, it follows, again by \eqref{1st-Lambda1SP}, that 	$\inf_{(c,u)\in (K\cap [0,c^*))\times S_0}\Lambda^-(c,u)>0$	 and hence the proof is complete.

\smallskip
	
\eqref{lSPc2}	In fact, suppose that $\inf_{(c,u)\in K\times S_0}\Lambda^+(c,u)>-\infty$. Since, by Lemma \ref{l1SP} \eqref{l1-3SP}, we know that $\sup_{(c,u)\in K\times S_0}t^+(c,u)<+\infty$, it follows by \eqref{1st-Lambda1SP} and \eqref{2nd-4} that $\inf_{(c,u)\in K\times S_0}t^+(c,u)>0$, $\inf_{u\in S_0}A(u)>0$ and $\inf_{u\in S_0}B(u)>0$, hence, by Lemma \ref{l1SP} \eqref{l1-3SP}, we obtain that  $\sup_{u\in S_0}N(u)<+\infty$. 
		
		Now assume that $\sup_{u\in S_0}N(u)<+\infty$ and $\inf_{u\in S_0}B(u)>0$.  The latter assertion yields, by \eqref{dp1SP}, that
		\begin{equation*}
			C_1t^+(c,u)^\eta-C_2t^+(c,u)^\beta-\alpha c \ge 0, \quad  \forall (c,u)\in K\times S_0,
		\end{equation*}
		which implies that $\inf_{(c,u)\in K\times S_0}t^+(c,u)>0$. Since by Lemma \ref{l1SP} \eqref{l1-3SP} we know that $\sup_{(c,u)\in K\times S_0}t^+(c,u)<+\infty$, the proof follows by \eqref{1st-Lambda1SP}.
\end{proof}

\begin{remark}\label{lSPcRmk} One can easily see from the proof that indeed
the stronger assertion holds:
		\begin{eqnarray*}
		\inf_{(c,u)\in (-\infty,c^*)\times S_0}\Lambda^-(c,u)>0 \ \ \Longrightarrow \ \ 
		\begin{cases}
		 \sup_{u\in (-\infty,c^*)\times S_0}t^-(c,u)<+\infty, \\
 \sup_{u\in S_0}N(u)<+\infty, \\
 \inf_{u\in S_0}A(u)>0, \\
 		 \inf_{u\in S_0}B(u)>0.
		\end{cases}
	\end{eqnarray*}
\end{remark}
Other properties of the maps $\Lambda^{\pm}$ are listed below.
\begin{lemma} \label{plSP} \strut
	The following  hold:
	\begin{enumerate}[label=(\roman*),ref=\roman*]
		\item \label{pl-0SP} $\Lambda^+(c,u)< \Lambda^-(c,u)$ for every $u \in S$ and $c\in (0,c^*)$. Moreover $\Lambda^-(c,u)>0$ for every $u \in S$ and $c< c^*$.
		\item\label{pl-iSP} The maps $u \mapsto \Lambda^-(c,u), \Lambda^+(c,u)$ are bounded from above on $S$, for every $c< c^*$ and $c\in (0,c^*)$, respectively.
		\item\label{pl-iiSP}  For every $u\in X \setminus\{0\}$ the maps $c\mapsto \Lambda^+(c,u),\Lambda^-(c,u)$ are increasing in $(0,c^*)$ and $(-\infty,c^*)$, respectively.
		\item\label{pl-iiiSP} $\displaystyle \lim_{c \to -\infty}\Lambda^-(c,u)= 0$ uniformly in $S$.
		\item\label{pl-ivSP} $\displaystyle \lim_{c \to 0^+}\Lambda^+(c,u)= -\infty$ uniformly in $S$.
	\end{enumerate}
\end{lemma}

\begin{proof}
	\eqref{pl-0SP} It follows from itens  \eqref{c-i} and \eqref{c-iv} of Lemma \ref{techn1}.
	
		\smallskip

	\eqref{pl-iSP} Suppose, on the contrary, that there exists $\{u_n\}\subset S$ such that $\Lambda^-(c,u_n)\to +\infty$. By Lemma \ref{lSPc} \eqref{lSPc1} we know that $\{t^-(c,u_n)\}$ is bounded and either $N(u_n)\to \infty$ or $A(u_n)\to 0$. Now, by \eqref{2nd-3} we know that $\{B(u_n)\}$ is bounded, so that from Lemma \ref{l1SP} \eqref{l1-1SP} we deduce that $\{N(u_n)\}$ is  bounded as well, and consequently $A(u_n) \to 0$. However, in such case we also have $B(u_n)\to 0$, which contradicts Lemma \ref{l1SP} \eqref{l1-1SP}. Thus  $u \mapsto \Lambda^-(c,u)$ is bounded from above on $S$. Since $\Lambda^+(c,u)< \Lambda^-(c,u)$ for all $c\in (0,c^*)$ and $u\in S$ we deduce that $u \mapsto \Lambda^+(c,u)$ is also bounded from above on $S$.
	
	\smallskip
	
	\eqref{pl-iiSP}  It follows from  Lemma \ref{l0} \eqref{I0-ii} and the fact that $I_2<0$.
	
	\smallskip
	
	\eqref{pl-iiiSP} 	 Suppose on the contrary that there exists a sequence $c_m\to -\infty$ and $u_m\in S$ such that $\Lambda^-(c_m,u_m)\ge C>0$, where $C$ is a positive constant. By Lemma \ref{lSPc} \eqref{lSPc1} and Remark \ref{lSPcRmk} we conclude that $t^-(c_m,u_m)$ is bounded, $B(u_m)$ is away from $0$ and $N(u_m)$ is bounded. Thus, by \eqref{1st-Lambda1SP}, we conclude that $\Lambda^-(c_m,u_m)\to \infty$ as $m\to \infty$, which is a contradiction since, by the last item, $\Lambda^-(c_m,u_m)<\Lambda^-(0,u_m)\le D$.

\smallskip

	\eqref{pl-ivSP} Note from \eqref{2nd-3} and \eqref{dp1SP} that there exist positive constants $C_1,C_2$ such that
	\begin{equation}\label{ineq}
		C_1t^+(c,u)^\eta-C_2t^+(c,u)^\beta-\alpha c\le 0,\quad \forall c\in (0,c^*),\ u\in S,
	\end{equation}
which implies that
\begin{equation*}
	C_1t^+(c,u)^{\eta-\alpha}-C_2t^+(c,u)^{\beta-\alpha}\le \alpha ct^+(c,u)^{-\alpha},\quad \forall c\in (0,c^*),\ u\in S,
\end{equation*}
and by Lemma \ref{l1SP} \eqref{l1-4SP} we conclude that $ct^+(c,u)^{-\alpha}\to \infty$ as $c\to 0^+$, uniformly in $S$. Now divide \eqref{ineq} to obtain
\begin{equation*}
	\frac{1}{c}(C_1t^+(c,u)^\eta-C_2t^+(c,u)^\beta)\le \alpha, \quad\forall c\in (0,c^*),\ u\in S. 
\end{equation*}
Since $\eta<\beta$ we conclude from Lemma \ref{l1SP} \eqref{l1-4SP} that $c^{-1}t^+(c,u)^\eta\to z<\infty$ as $c\to 0^+$, uniformly in $S$, and then $c^{-1}t^+(c,u)^\beta=c^{-1}t^+(c,u)^\eta t^+(c,u)^{\beta-\eta}\to 0$ as $c\to 0^+$, uniformly in $S$. By plugging these informations in the third equality of \eqref{1st-Lambda1SP} we obtain the desired conclusion.
\end{proof}
\medskip
In the sequel $I_+=(0,c^*)$ and $I_-=(-\infty,c^*)$.
\begin{proposition}\label{PSSP} 
	Suppose that $c_n\to c \in \mathcal{I} _\pm$
	and $\{u_{n}\}\subset S$ are such that $\frac{\partial \widetilde{\Lambda^{\pm}}}{\partial u}(c_n,u_n)\to 0$ and $\widetilde{\Lambda^{\pm}}(c_n,u_n)\to \mu$ ($>0$ if $\widetilde{\Lambda^{-}}$ is considered). Then $\{u_n\}$ has  a convergent subsequence. In particular, the map $u \mapsto \widetilde{\Lambda^{\pm}}(c,u)$ satisfies the Palais-Smale condition at every $\mu$ ($>0$ if $\widetilde{\Lambda^{-}}$ is considered). 
\end{proposition}
\begin{proof} Write, for simplicity, $\Lambda:=\Lambda^{\pm}$ and $t(c,u):=t^\pm(c,u)$.  Let $t_n:=t(c_n,u_n)$. 	We can assume that  $u_n\rightharpoonup u$ in $X$. By Lemma \ref{l1SP} items \eqref{l1-1SP} and \eqref{l1-3SP} and Lemma \ref{lSPc} we 
	find that $d_1\le t_n\le d_2$ for all $n$ and some $d_1,d_2>0$. Similarly, one can prove that $u\neq 0$. We use then Lemma \ref{psg} and Remark \ref{rps} to conclude.
\end{proof}

By Lemmae \ref{lem:H1SP} and \ref{plSP} and Proposition \ref{PSSP} we are able to apply Theorem \ref{THM1}
to show that for any $c$ the sequences
$$\mu_{n,c}^- = \sup_{F\in \mathcal F_{n}} \inf_{u\in F}\widetilde{\Lambda^-}(c,u),\quad \mbox{if} \quad c<c^*,$$ and 
$$ \mu_{n,c}^+ = \sup_{F\in \mathcal F_{n}} \inf_{u\in F}\widetilde{\Lambda^+}(c,u), \quad \mbox{if} \quad  c\in(0,c^*).$$
give rise to critical points of the families of functionals $\Phi_{\mu_{n,c}^{-}}$ and $\Phi_{\mu_{n,c}^{+}}$ having energy equal to $c$. From now, if  necessary, we will use the fact that (see Remark \ref{infsup})
	$$\mu_{n,c}^-= \left(\sup_{F\in \widetilde{\mathcal{F}}_n}\inf_{u\in F}(\Lambda^{-})^{-1}(c,u)\right)^{-1},$$
and
$$\mu_{n,c}^+=- \sup_{F\in \widetilde{\mathcal{F}}_n}\inf_{u\in F}-\Lambda^{+}(c,u).$$

\begin{corollary}\label{ciiSP}
	There holds $\mu_{n,c}^- \to 0$ as $n\to \infty$ for any $c<c^*$, while $\mu_{n,c}^+ \to -\infty$ as $n\to+ \infty$ for any $c\in(0,c^*)$. Moreover, $\|u_{n,c}\| \to +\infty$ as $n\to +\infty$. Finally, if $N(v_{n,c}/\|v_{n,c}\|)$ is bounded  then $v_{n,c} \rightharpoonup 0$ and $v_{n,c} \not \to 0$ in $X$ as $n\to \infty$, while if $N(v_{n,c}/\|v_{n,c}\|)$ is unbounded, then up to a subsequence, $v_{n,c}\to 0$ as $n\to+ \infty$.
\end{corollary}	

\begin{proof}
	By Lemma \ref{lSPc} we know that $(\Lambda^{-})^{-1}(c,w_n) \to +\infty$ and $-\Lambda^{+}(c,w_n)\to +\infty$ if $w_n \rightharpoonup 0$ in $X$, so Theorem \ref{THM1} yields the first assertion. Moreover, writing $u_{n,c}=t^-(c,w_n)w_n$ with $w_n \in S$, we have $\|u_{n,c}\|=t^-(c,w_n)$ and $\Lambda^-(c,w_n)=\mu_{n,c}^- \to 0$. By Lemma \ref{lSPc} we deduce that $t^-(c,w_n)\to+ \infty$. A similar argument yields $v_{n,c}=t^+(c,w_n)w_n$ with $\Lambda^+(c,w_n)=\mu_{n,c}^+ \to -\infty$, so by Lemma \ref{lSPc} we have either $N(w_n) \to +\infty$ or $A(w_n) \to 0$. If $N(w_n)$ is unbounded then by Lemma \ref{l1SP} \eqref{l1-3SP} we obtain that $t^+(c,w_n)\to 0$ as $n\to +\infty$. If $N(w_n)$ is bounded then $w_n \rightharpoonup 0$ in $X$. Since $t^+(c,w_n)$ is bounded by Lemma \ref{l1}, we see that $v_{n,c} \rightharpoonup 0$ in $X$. Finally, \eqref{dp1} shows that $t^+(c,w_n) \not \to 0$, so $v_{n,c} \not \to 0$ in $X$.
\end{proof}
Let us proceed with the study of $\mu_{n,c}^\pm$ with respect to $c$.

\begin{lemma}\label{lipSP} 
	For every $n$ the following properties hold:
	\strut
	\begin{enumerate}[label=(\roman*),ref=\roman*]
		\item\label{lip-0SP}  The maps $c \mapsto \mu_{n,c}^\pm$ are increasing and locally Lipschitz continuous in $I_\pm$. 
		\item\label{lip-iSP} $\displaystyle \lim_{c\to -\infty}\mu_{n,c}^-=0$ and $\displaystyle \lim_{c\to 0^+}\mu_{n,c}^+=-\infty$.
	\end{enumerate}
\end{lemma}
\begin{proof} 
	\eqref{lip-0SP} Indeed, by \eqref{dl} and Lemma \ref{lSPc} we see that \eqref{H3} is satisfied. Theorem \ref{THM2} yields the conclusion.
	
		\smallskip

	\eqref{lip-iSP} It follows from Lemma \ref{plSP} items \eqref{pl-iiiSP} and \eqref{pl-ivSP}.
\end{proof}

Now we prove a multiplicity result for fixed $\mu$.

\begin{proposition}\label{ncriticaSP} Given $n\in \mathbb{N}$, for any $\mu\in(0,\mu^-_{n,0})$ the functional $\Phi_{\mu}$ has at least $n$ pairs of critical points with negative energy. 
\end{proposition}
\begin{proof} Indeed, by Lemma \ref{lipSP} \eqref{lip-0SP} there exists $c\in(-\infty,0)$ such that $\mu^-_{n,c}=\mu$, so that $\Phi_{\mu}$ has a pair of critical points with negative energy. In a similar way, there exists $c_1$ such that $\mu^-_{n-1,c_1}=\mu$. Therefore $\Phi_{\mu}$ has another pair of critical points with negative energy. By repeating this procedure we obtain $n$ pairs of critical points with negative energy.  
\end{proof}
\begin{proposition}\label{ncriticaSP1} 
Fix $c\in (0,c^*)$. 
Then for any $\mu<\mu^+_{n,c}$ the functional $\Phi_{\mu}$ has at least $n$ pairs of critical points with positive energy. 
\end{proposition}
\begin{proof} Indeed, by Lemma \ref{lipSP} the vertical line at $(\mu,0)$ crosses the curves $(0,c^*)\ni c\mapsto \mu_{1,c}, \cdots, \mu_{n,c}$, which complete the proof.
	\end{proof}
To conclude let us prove 
\begin{proposition} \label{pdmSP} There holds $\mu_{n,c}^+<\mu_{n,c}^-$ for every $n\in \mathbb{N}$ and $c\in (0,c^*)$. 
\end{proposition}
The proof of this result relies on the following lemma:

\begin{lemma}\label{seconddSP} Suppose that there exists a sequence $(s_n,u_n)\in \mathbb{R}\times S$ such that $t^+(c,u_n)\le s_n<t^+(c,u_n)+1/n$, $\Lambda^+(c,u_n)$ is bounded from below and $\psi_{c,u_n}''(s_n)=o_n(1)$. Then $c\ge c^*$.
\end{lemma}
\begin{proof} Let us write $t_n=t^+(c,u_n)$. By Lemmae \ref{l1SP} and \ref{lSPc} we know that $\{t_n\}$ and $\{A(u_n)\}$ are bounded and away from zero, so we can assume that $t_n,s_n \to t>0$ and $u_n \rightharpoonup u\neq 0$. Thus
\small{	\begin{eqnarray*}
			\frac{\alpha-\eta}{\eta}N(u_n)s_n^{\eta}-\frac{\alpha-\beta}{\beta}B(u_n)s_n^{\beta}-\alpha c 
		&=&\frac{\alpha-\eta}{\eta}N(u_n)(t_n+o_n(1))^{\eta}-\frac{\alpha-\beta}{\beta}B(u_n)(t_n+o_n(1))^{\beta}-\alpha c \\
		&=& t_n^{\alpha+1}A(u_n)\psi_{c,u_n}'(t_n)+o_n(1)=o_n(1),\ \forall n,
	\end{eqnarray*}}
	which combined with $\psi_{c,u_n}''(s_n)=o_n(1)$ yields
	\begin{equation*}
		\label{pu}
		\left\{
		\begin{array}
			[c]{lll}%
			\displaystyle\frac{\alpha-\eta}{\eta}N(u_n)s_n^{\eta}-\frac{\alpha-\beta}{\beta}B(u_n)s_n^{\beta}-\alpha c =o_n(1),\\
						\displaystyle\frac{(\alpha-\eta)(\eta-\alpha-1)}{\eta}N(u_n)s_n^{\eta}-\frac{(\alpha-\beta)(\beta-\alpha-1)}{\beta}B(u_n)s_n^{\beta}+\alpha(\alpha+1) =o_n(1).
		\end{array}
		\right. 
	\end{equation*}
	Solving this system in the variables $(s_n,c)$ we conclude that
	$s_n=\left(\frac{\alpha-\eta}{\alpha-\beta}\frac{N(u_n)}{B(u_n)}+o_n(1)\right)^{\frac{1}{\beta-\eta}}$,
	and (recall \eqref{RayleighSP}) $c=\lim_{n\to+\infty} c(u_n)\ge c^*$.
\end{proof}
As a consequence we have
\begin{corollary}\label{secondd1SP} Fix $c\in (0,c^*)$ and suppose that $\Lambda^+(c,\cdot)$ is bounded from below over $S_0\subset S$. Then there exists $D,\delta>0$ such that $\psi_{c,u}''(s)\ge D$ for all $(s,u)\in[t^+(c,u),t^+(c,u)+\delta]\times S_0$.
\end{corollary}
\begin{proof} On the contrary we can find a sequence $(s_n,u_n)\in \mathbb{R}\times S_0$ such that $t^+(c,u_n)\le s_n<t^+(c,u_n)+1/n$, $\Lambda^+(c,u_n)$ is bounded and $\psi_{c,u_n}''(s_n)=o_n(1)$, which contradicts Lemma \ref{seconddSP}.
\end{proof}

\begin{proof}[Proof of Proposition \ref{pdmSP}] By using Lemma \ref{seconddSP} and Corollary \ref{secondd1SP} the proof is similar to that of Proposition \ref{pdm}.
\end{proof}
\begin{theorem}\label{THMSPAbstrac} Under 
condition \eqref{H4} we have that 
		\begin{enumerate}[label=(\roman*),ref=\roman*]
		\item \label{THMSPAbstrac-i} For any $c<c^*$ there exists a sequence $\{(\mu_{n,c}^-,u_{n,c})\}  \subset (0,+\infty)\times X\setminus\{0\}$ such that $\Phi_{\mu_{n,c}^-}(\pm u_{n,c})=c$ and $\Phi'_{\mu_{n,c}^-}(\pm u_{n,c})=0$, i.e., $\pm u_{n,c}$ are critical points of $\Phi_{\mu}$ with $\mu=\mu_{n,c}^-$, having energy $c$, for every $n$.  Moreover:
		\begin{enumerate}
			\item  $\mu_{n,c}^-$ is non-increasing,  $\displaystyle \lim_{n\to +\infty}\mu_{n,c}^-=0$ and $\|u_{n,c}\|\to+ \infty$ as $n\to +\infty$, so $(0,+\infty)$ is a bifurcation point.
			\item If $c<0$ and $\mu>\mu_{1,c}^-$ then $\Phi_\mu$ has no critical points having energy $c$.\\
		\end{enumerate}
		
		\item \label{THMSPAbstrac-ii} For any $c\in (0,c^*)$  there exists a sequence $\{(\mu_{n,c}^+,v_{n,c})\}  \subset \mathbb{R} \times X\setminus\{0\}$ such that $\Phi_{\mu_{n,c}^+}(\pm v_{n,c})=c$ and $\Phi'_{\mu_{n,c}^+}(\pm v_{n,c})=0$, i.e., $\pm v_{n,c}$ are critical points of $\Phi_\mu$ with $\mu=\mu_{n,c}^+$, having energy $c$, for every $n$. Moreover:
		\begin{enumerate}
			\item $\mu_{n,c}^+$ is non-increasing, $\displaystyle \lim_{n\to +\infty}\mu_{n,c}^+=-\infty$ and $v_{n,c} \rightharpoonup 0$ as $n\to +\infty$.
			\item $\mu_{n,c}^+<\mu_{n,c}^-$ for every $n$.
		\end{enumerate}
		\item \label{THMSPAbstrac-iii} The map $c \mapsto \mu_{n,c}^-$ is continuous and increasing in $(-\infty,c^*)$, and $\displaystyle \lim_{c\to -\infty}\mu_{n,c}^-=0$. 
	\item \label{THMSPAbstrac-iv} The map $c \mapsto \mu_{n,c}^+$ is continuous and increasing in $(0,c^*)$, and $\displaystyle \lim_{c\to 0^+}\mu_{n,c}^+=-\infty$. 
	\item \label{THMSPAbstrac-v} For every $\mu\in (0,\mu_{n,0}^-)$ the functional $\Phi_\mu$ has at least $n$ pairs of critical points with negative energy. 
	\item\label{THMSPAbstrac-vi}For every $\mu< \mu_{n,0}^+$ the functional $\Phi_\mu$ has at least $n$ pairs of critical points with positive energy. 
	\end{enumerate}
\end{theorem}
\begin{proof}\eqref{THMSPAbstrac-i} and \eqref{THMSPAbstrac-ii} follow from Corollary \ref{ciiSP}, Lemma \ref{plSP}, and Propositions \ref{PSSP} and \ref{pdmSP}, combined with Theorems \ref{THM0} and \ref{THM1}, whereas 
\eqref{THMSPAbstrac-iii} and \eqref{THMSPAbstrac-iv} follow from Lemma \ref{lipSP}. Finally, 
\eqref{THMSPAbstrac-v} and \eqref{THMSPAbstrac-vi} are consequences of Propositions \ref{ncriticaSP} and \ref{ncriticaSP1}.
\end{proof}
\begin{remark}\label{r1} \rm
	The previous result also holds if instead of assuming that $A'$ is completely continuous, we assume that $A=N^{\sigma}$ for some $\sigma>0$. 
\end{remark}
\begin{remark}\label{conditon35} \rm
	We note here that condition \eqref{H4} is not a pure technical condition. In fact,  in \cite{Ru},
	where the case $a=0$ is treated (the so called Schr\"odinger-Poisson system), for $p=3$
	the functional satisfies all conditions except \eqref{H4}. Moreover $\inf_{u\in X} \Phi_{\mu}(u)=-\infty$ for small $\mu>0$, while  $\inf_{u\in X} \Phi_{\mu}(u)=0$ for large $\mu>0$. Also, in \cite[Theorem 1.6]{Ru} we have a functional $\Phi_\mu$ satisfying all conditions of this section, except for \eqref{2nd-3} and \eqref{H4} and similarly it is not bounded from below for small $\mu$, while its infimum is zero for large $\mu$. 
%
 
 Moreover, from Lions inequality (see \cite{L}) the functional $\Lambda(0,\cdot)$ is bounded from above when $p=3$ and consequently $\Lambda(c,\cdot)\le \Lambda(0,\cdot)$ is bounded for all $c\le 0$. However, it is an open question whether its supremum $\mu_{1,0}$ is achieved or not. In fact, we can obtain more from \cite{Ru}: since for $p=3$ and $\mu<\mu_{1,0}$, the functional $\Phi_\mu$ has no critical points with negative energy, it follows that $\widetilde{\Lambda}(c,\cdot)$ does not satisfies the Palais-Smale condition for  $c<0$. 
\end{remark}

\subsection{Proof of Theorems \ref{THMAP3} and \ref{EC3}}
The following lemma will be useful in proving the results.
\begin{lemma}Condition \eqref{H4} holds true.
\end{lemma}
\begin{proof} Indeed, suppose that $\Phi_\mu(u)\le c$ with $\mu\ge a$, set $D_\mu:=\frac{\mu}{16\pi}-\varepsilon^4$ and note that
\begin{eqnarray}\label{M1}
	c\ge\Phi_\mu(u)&=&\frac{1}{4}\|\nabla u\|_2^2+\frac{1}{4}\|\nabla u\|_2^2+\frac{\omega}{2}\| u\|_2^2+\frac{\mu}{4}\int_{\mathbb{R}^3} \phi_u u^2-\frac{1}{p}\|u\|_p^p \nonumber\\
	&\ge& \frac{1}{4}\|\nabla u\|_2^2+D_\mu\|\phi_u\|_\mathcal{D}^2+\frac{\omega}{2}\| u\|_2^2+\frac{\pi\varepsilon^2}{4}\|u\|_3^3-\frac{1}{p}\|u\|_p^p\nonumber\\
	&=&\frac{1}{4}\|u\|^2+D_\mu\|\phi_{u}\|_{\mathcal D}^{2}+\int_{\mathbb{R}^3} h(u), 
\end{eqnarray}
where
\begin{equation*}\label{ft}
	h(t)=\frac{\omega}{4}t^2+\frac{\pi\varepsilon^2}{4}t^3-\frac{1}{p}t^p,\quad \forall\, t>0.
\end{equation*}
We can choose $\varepsilon>0$ such that $D_\mu>0$ for all $\mu \geq a$. Now, suppose on the contrary that there exists a sequence $(\mu_n,u_n)\in [a,+\infty)\times H_r^1(\mathbb{R}^3)$ such that $u_n\to \infty$ as $n\to \infty$. A simple analysis shows that $I:=\inf_{t>0}h(t)>-\infty$ and if $h(t)<0$ for some $t>0$, then $h^{-1}((-\infty,0))=(d_1,d_2)$, where $0<d_1<d_2<+\infty$. If $I:=\inf_{t>0}h(t)\ge 0$ we clearly have a contradiction with \eqref{M1}, thus we can assume that $I<0$.

Write $A_n=\{ x\in \mathbb R^{3}: u_n(x)\in (d_1,d_2)\}$ and $\rho_n=\sup\{|x|:\ x\in A_n\}$. Arguing as in the proof of \cite[Proposition 3.1]{SS}, we conclude that $\operatorname{meas}(A_n)\to+ \infty$ as $n\to \infty$,
\begin{equation}\label{H41}
	\frac{C_1\rho_n^2-c}{|I|}\le \operatorname{meas}(A_n),\quad \forall n\in \mathbb{N},
\end{equation}
and
\begin{equation}\label{H42}
\frac{|I|\operatorname{meas}(A_n)+c}{d_1^{4}}\ge \frac{1-e^{-\frac{2\rho_n}{a}}}{2\rho_n}\operatorname{meas}(A_n)^2, \quad\forall n\in \mathbb N.
\end{equation}
Since $\operatorname{meas}(A_n)\to +\infty$ we conclude from \eqref{H42} that $\rho_n\to+ \infty$. Moreover, by combining \eqref{H41} and \eqref{H42} we obtain that
\begin{equation*}
	\frac{|I|}{d_1^{4}}+\frac{c}{\operatorname{meas}(A_n)}\ge \frac{1-e^{-\frac{2\rho_n}{a}}}{2\rho_n}	\frac{C_1\rho_n^2-c}{|I|},\quad \forall n\in \mathbb N,
\end{equation*}
which is a contradiction, therefore \eqref{H4} is true.
\end{proof}

Finally we can prove Theorem \ref{THMAP3} and Theorem \ref{EC3}. Indeed, one may check that 
$N(u)=\int_{\mathbb{R}^3} \left(|\nabla u|^2+\omega |u|^2\right)$, $A(u)=\int_{\mathbb{R}^3} \phi_uu^2$, and $B(u)=\int_{\mathbb{R}^3} |u|^p$, defined on $X=H_r^1(\mathbb{R}^3) \setminus \{0\}$, satisfy the conditions \eqref{2nd-1}-\eqref{2nd-6} of this section with $\eta=2$, $\alpha=4$, and $\beta=p$.

\appendix
\section{}
The next result holds under \eqref{H1} and the definitions related to it.
\begin{lemma} 
\label{psg}	
	 Suppose that $(c_n,u_n)\in \mathcal{I}\times S$ and the functional $\Phi_{\mu}$ satisfy
	\begin{enumerate}[label=(\roman*),ref=\roman*]
		\item\label{PSS1} $c_n\to c\in \mathcal{I}$, $\widetilde{\Lambda}(c_n,u_n)\to \mu$  and $\frac{\partial \widetilde{\Lambda}}{\partial u_n}(c_n,u_n)\to 0$,
		\item\label{PSS2} $t(c_n,u_n) \to t>0$ and $u_n \rightharpoonup u \neq 0$,
		\item\label{PSS3} $I_2$ is bounded on bounded sets,
		\item\label{PSS4} if $\mu_k\to \mu$, $z_k\rightharpoonup z$ and $\Phi_{\mu_k}'(z_k)\to 0$, then $z_k\to z$.
	\end{enumerate}
Then  $u_n \to u$.
\end{lemma}
\begin{proof}  
Given $w\in X$ and  $n\in \mathbb N$,  let $(s_n,v_n)\in \mathbb{R}\times \mathcal{T}_{S}(u_n)$ 
be the unique pair such that $w=v_n+s_nu_n$. Hence $i'(u_n)w=s_ni'(u_n)u_n=s_n$, so $\{s_n\}$ is bounded, and consequently $\{v_n\}$ is bounded as well. Thus $\frac{\partial \widetilde{\Lambda}}{\partial u}(c_n,u_n)v_n\to 0$ and since $\frac{\partial \Lambda}{\partial u}(c_n,u_n)u_n=0$, we conclude that 

\begin{equation}\label{PSAL}
	\frac{\partial \Lambda}{\partial u}(c_n,u_n)w=
\frac{\partial \widetilde{\Lambda}}{\partial u}(c_n,u_n)v_n+s_n\frac{\partial \Lambda}
{\partial u}(c_n,u_n)u_n\to 0.
\end{equation}
Recall from \eqref{lp} that
	\begin{equation} \label{PSAL1}
\Phi_{\Lambda(c,u)}'(t(c,u)u)=\frac{I_2(t(c,u)u)}{t(c,u)}\frac{\partial \Lambda}{\partial u}(c,u), \qquad \forall (c,u) \in \mathcal{I} \times X\setminus\{0\}.
\end{equation}
It follows from \eqref{PSS1},  \eqref{PSS2}, \eqref{PSS3}, \eqref{PSAL} and \eqref{PSAL1} that $\Phi_{\Lambda(c_n,u_n)}'(t(c_n,u_n)u_n)\to 0$ and thus, by \eqref{PSS1}, \eqref{PSS2} and \eqref{PSS4}, $t(c_n,u_n)u_n \to tu$, which yields the conclusion.
\end{proof}

\begin{remark}\label{rps}
Standard arguments (relying in particular on the uniform convexity of $X$) show that condition \eqref{PSS4} in the previous Lemma is satisfied if $\Phi_\mu=N-I_{\mu}$, where $N,I_\mu \in C^1(X)$ are functionals satisfying the following conditions:
\begin{enumerate}[label=(\roman*),ref=\roman*]	
	\item\label{1st-AP} for every $\mu \in \mathbb{R}$ the functional $I_\mu'$ is completely continuous, i.e. $I_{\mu}'(u_n) \to I_\mu'(u)$ in $X^*$ if $u_n \rightharpoonup u$ in $X$. Moreover, for any $u \in X$ the map $\mu \mapsto I_\mu(u)$ is continuous.
	\item\label{2st-AP} $N$ is weakly lower semicontinuous and there exist $C,\eta>0$ such that 
	$$(N'(u)-N'(v))(u-v)\geq C(\|u\|^{\eta-1}-\|v\|^{\eta-1})(\|u\|-\|v\|)$$ for any $u,v \in X$.\\
\end{enumerate}
\end{remark}

\medskip

{\bf Acknowledgement.}
G. Siciliano was partially supported by Fapesp grant 2019/27491-0, CNPq grant 304660/2018-3
, FAPDF, and CAPES (Brazil) and INdAM (Italy).
K. Silva was partially supported by CNPq/Brazil under Grant 408604/2018-2.

\end{document}